\documentclass{article}

\usepackage{lineno,hyperref}
\modulolinenumbers[5]

\usepackage{amsfonts}

\usepackage{amsmath}
\usepackage{amssymb}
\usepackage{amsthm}

\usepackage{multirow}

\usepackage{algorithm}
\usepackage{algpseudocode}
\algnewcommand{\LeftComment}[1]{\Statex \(\triangleright\) #1}
\usepackage[percent]{overpic}

\usepackage[capitalize]{cleveref}

\usepackage{tikz,tikz-cd}

\theoremstyle{plain}
\newtheorem{theorem}{Theorem}[section]

\newtheorem{proposition}[theorem]{Proposition}
\newtheorem{lemma}[theorem]{Lemma}

\theoremstyle{definition}
\newtheorem{definition}[theorem]{Definition}
\newtheorem{remark}[theorem]{Remark}
\newtheorem{example}[theorem]{Example}

\newcommand{\SimpComplex}{\mathsf{SimpComplex}}
\newcommand{\Vect}{\mathsf{Vec}}
\newcommand{\ChainComplex}{\mathsf{ChComplex}}

\newcommand{\bbN}{\mathbb{N}}
\newcommand{\bbR}{\mathbb{R}}
\newcommand{\bbZ}{\mathbb{Z}} 

\renewcommand{\to}{\rightarrow}

\newcommand{\iso}{\cong}

\renewcommand{\Im}{\mathrm{Im}}
\newcommand{\Ker}{\operatorname{Ker}}
\newcommand{\Coker}{\operatorname{Coker}}

\newcommand{\define}[1]{\textbf{\boldmath{#1}}}
\newcommand{\bfz}{\textbf{z}}
\newcommand{\bfw}{\textbf{w}}
\newcommand{\bfp}{\textbf{p}}
\newcommand{\gr}{\mathsf{gr}}
\newcommand{\mpfree}{\texttt{mpfree}}
\newcommand{\multichunk}{\texttt{multi-chunk}}

\font\myfont=cmr12 at 16pt

\title{{\myfont Compression for 2-Parameter Persistent Homology}\thanks{This work was supported by the Austrian Science Fund (FWF) grant
numbers P 29984-N35 and P 33765-N.}}

\author{Ulderico Fugacci\thanks{Istituto di Matematica Applicata e Tecnologie Informatiche ``E. Magenes'', Consiglio Nazionale delle Ricerche, Italy} 
\and Michael Kerber\thanks{Graz University of Technology, Austria} 
\and Alexander Rolle\footnotemark[3]}

\date{}

\begin{document}

\maketitle
 
\begin{abstract}
    Compression aims to reduce the size of an input, while maintaining its relevant properties. 
    For multi-parameter persistent homology, 
    compression is a necessary step in any computational pipeline, 
    since standard constructions lead to large inputs, 
    and computational tasks in this area tend to be expensive.
    We propose two compression methods for chain complexes of free 2-parameter persistence modules.
    The first method extends the  multi-chunk algorithm for one-parameter persistent homology, 
    returning the smallest chain complex among all the ones quasi-isomorphic to the input. 
	The second method produces minimal presentations of the homology of the input; 
	it is based on an algorithm of Lesnick and Wright, 
	but incorporates several improvements that lead to substantial performance gains. 
	The two methods are complementary, and can be combined to compute minimal presentations for complexes with millions of generators 
	in a few seconds. 
    The methods have been implemented, and the software is publicly available. 
    We report on experimental evaluations, which demonstrate substantial improvements in performance 
    compared to previously available compression strategies. 
 \end{abstract}

\section{Introduction}
\label{sec:intro}

\paragraph{Motivation}
Persistent homology is a powerful tool in data analysis enabling the characterization and the comparison of a large variety of datasets on the basis of their topological properties. 
Specifically, the typical pipeline of persistent homology assigns to each data set a filtered simplicial complex evolving along a real-valued parameter. Then, persistent homology describes and discriminates data by studying the birth and the death of the topological features (such as connected components, loops, tunnels, and so on) occurring during the filtration.

While this strategy has proven its effectiveness in several scenarios, it is quite common in applications to face datasets that are parametrized through multiple scalar functions. This naturally calls for the development of a multi-parameter extension of persistent homology.
Despite its potential, multi-parameter persistent homology has not found as widespread distribution as the one-parameter case:
a major obstruction is the impossibility of summarizing the topological information by a complete and discrete invariant such as the barcode
in the one-parameter setting~\cite{cz-theory} (which encodes the births and deaths mentioned above). 
Still, recent work has proposed algorithmic solutions to visualize multi-parameter persistence modules~\cite{lw-multi}, to compare them
via the matching distance~\cite{klo-exact, kn-matching, bk-matching}, the bottleneck distance~\cite{dx-computing} or via kernels~\cite{cfklw-kernel,vipond-landscape}, and to decompose them into indecomposable summands~\cite{dx-generalized}.
What is common to these approaches is that the computation tends to be expensive. For instance, the algorithm to compute the matching distance
has expected running time 
%$O(n^{11})$ 
$O(n^{5} \log^{3} n)$ for $2$-parameters, where 
the worst-case complexity for a single parameter is only $O(n^{1.5}\log n)$~\cite{kmn-geometry}.
So despite all efforts, distance computations in the $2$-parameter case
must be restricted in practice to modules of moderate size. 

By \textbf{compression} of a representation of a multi-parameter topological object 
(e.g. a filtered simplicial complex, filtered chain complex, or a persistence module) 
we mean replacing the representation with a smaller representation such that the relevant topological properties are preserved. 
Computations on compressed representations yield the same outcome by definition, and will perform
faster because of the reduced input size. Hence, compression has the potential to improve the performance of each of the aforementioned approaches.
We therefore consider compression as an indispensable pre-processing step of the computational pipeline of multi-parameter topological data analysis.

\paragraph{Contribution}
This paper proposes two compression algorithms for 2-pa\-ram\-e\-ter persistent homology.
Both algorithms take as input chain complexes of free 2-parameter persistence modules; 
as an example, bi-filtered simplicial complexes naturally yield such chain complexes
with simplices as free generators (see \cref{chain-complexes-from-simplicial-complexes}). 
Both our approaches return a structure
that is minimal in an appropriate sense, but they differ in how much topological information they preserve.

The first approach, called \multichunk, computes another chain complex of free 2-parameter persistence modules
with the property that it is the smallest such chain complex among all those quasi-isomorphic to the input.
The algorithm is an extension of the \emph{chunk} algorithm for one-parameter persistent homology~\cite{chunk}: it proceeds by first identifying
pairs of generators in consecutive dimensions that do not contribute to the homology of the complex,
and then changing the boundaries of the generators to remove all those pairs.
Both steps can be performed quickly in practice using similar speed-ups as in the one-parameter counterpart.
Moreover, both steps can easily be parallelized with shared-memory.

The second approach, called \mpfree, computes a \emph{minimal presentation matrix} for the homology of the chain complex in a chosen degree. 
A presentation matrix encodes the generators and relations of a module. It is minimal
if the total numbers of generators and relations is minimal. 
%A minimal presentation is a more compressed representation of the persistence module
%compared to the outcome of multi-chunk, to the price that the chain complex structure is lost.
A minimal presentation of the homology is typically smaller than the chain complex produced by multi-chunk, 
with the price that the chain complex structure is lost. 
Some important algorithms in multi-parameter persistence require a presentation 
as input \cite{lesnick2019computing, dx-generalized}, 
and this is further motivation for computing minimal presentations efficiently.

Our algorithm \mpfree \, is an improved version of the pioneering work of Lesnick and Wright~\cite{lesnick2019computing}
which arrives at a minimal presentation matrix through column additions.
Our improvements stem from two major observations: first, their algorithm often spends more time
in \emph{looking} for the next column addition to be performed 
than it does in \emph{performing} the addition. 
We show that the necessary column operations can be predicted from previous steps, and we maintain this information in priority queues
to avoid linear scans. Second, we observe that the last step of the algorithm (which turns a so-called
semi-minimal presentation into a minimal one) can be done more efficiently using \multichunk.

Both approaches have been implemented in C++ and the software is publicly available~\cite{multi-chunk-repo, mpfree-repo}. 
We extensively benchmark our implementations on a variety of data sets including triangular mesh data,
bi-filtered flag complexes and others. We show that computing minimal chain complexes and minimal presentations scales linearly in terms of time and memory size in practice and is possible
for complexes with millions of simplices within seconds. 
In particular, our improvements lead to a significant improvement over 
the RIVET library~\cite{rivet}, 
which implements the algorithm by Lesnick and Wright.
Moreover, the resulting compressed outputs are significantly smaller than the
input and thus result in a satisfying compression rate.

We also show that \multichunk \, generally performs
faster than \mpfree, and its running time is dominated by reading the input file. More speed-ups can be achieved by first calling \multichunk, followed by \mpfree \, on the
compressed chain complex. Partially, the advantage of that strategy comes from the \emph{clearing optimization}~\cite{ck-twist,chunk}
that avoids some operations in dimension $d$ through the information obtained in dimension $d+1$. Since minimal presentations
focus on a single dimension, the advantage of clearing is mostly lost when using \mpfree \, alone.

Our experiments also show a (modest) advantage of parallelization, and 
we investigate the running time of the major subroutines of our algorithm.
Finally, we report on a perhaps counter-intuitive observation
that the algorithm is sometimes faster when applied to a chain complex
$C_3\to C_2\to C_1\to C_0\to 0$ than when only applied to the sub-complex
$C_2\to C_1\to C_0$.

\paragraph{Differences to the conference versions}
Both presented approaches appeared as prior conference papers; see~\cite{fugacci2019chunk} for \multichunk \,
and~\cite{kerber2021fast} for \mpfree. We think that because both approaches share the goal of compression
and also are symbiotic on a mathematical and implementation level, it is natural to present
them together in this journal version. 
Furthermore, in the conference versions, the two algorithms were described using different 
frameworks; in this paper we describe the two algorithms within the same framework 
(both take as input chain complexes of free, multi-parameter persistence modules). 
By unifying the language, we clarify the relationship between the two algorithms. 
Notably, we can describe
one sub-routine of \mpfree \, (the minimization) as an instance of
\multichunk, which was not clear in the conference version. 
We also provide a new complexity analysis for \mpfree, 
which sheds light on the improvement in performance over RIVET 
that we observe in the experiments. 
Moreover, we provide proof details that had to be skipped
in the conference version because of size limits.
Finally, this paper contains experimental evaluation for new benchmark instances
and investigates the interplay of both algorithms in the case of ``long'' chain complexes.

\paragraph{Related work}

Compression has been studied extensively for the one-parameter case in the context of persistence computation.
Besides the aforementioned chunk algorithm~\cite{chunk}, several papers propose the use of discrete Morse theory to collapse
pairs of simplices without changing the homotopy type, reducing
the size of the considered boundary matrix~\cite{gunther-memory,mn-morse}. 
Also, the efficiency of state-of-the-art software for flag complexes, such as Ripser~\cite{ripser} and Eirene~\cite{eirene}
is to a large extent due to implicit compression.
Another recent line of research uses the
concept of strong collapses to arrive at a small complex quickly, in general~\cite{bpp-strong}, and for flag complexes~\cite{bp-edge}.

All the above approaches are examples of \emph{lossless compression}, in the sense that the compressed complex has an equivalent persistence diagram to the input. 
Better compression rates are possible with \emph{lossy compression}. Here, it is only guaranteed that the compressed
complex has a persistence diagram that is close to the input. This idea has been extensively studied in theory~\cite{sheehy-linear,dfw-computing,bs-approximating,ckr-polynomial,ckr-improved} and practice~\cite{bcos-efficient,dsw-simba,sparips}.

In the multi-parameter case, one line of research lifts the idea of discrete Morse collapses to the multi-parameter case~\cite{Allili2017,Allili2018,Scaramuccia2018,Iuricich2016}. Our \multichunk \, algorithm is motivated by this line of research and attempts to compress the complex on the level of persistence
computation directly instead of using discrete Morse theory. This leads to a provably minimal chain complex which can be guaranteed in the
aforementioned approach only for special cases. We point out, however, that the related work returns a Forman gradient together with the cell complex,
and applications might profit from this additional structure. Finally, we compared our \multichunk \, algorithm with the approach in~\cite{Scaramuccia2018}: \multichunk \, is faster, and both approaches have comparable memory consumption and produce compressions
of similar size in practice. We refer to the the conference version of \multichunk \, \cite{fugacci2019chunk} for more details on this comparison.

Minimal presentations can be computed for
more general (graded) modules, and general algorithms are provided by computer algebra systems like
Macaulay or Singular. Lesnick and Wright~\cite{lw-multi} compare their implementation with these systems and show
that their algorithm is much faster for the special case of chain complexes of free 2-parameter persistence modules.

\paragraph{Outline}
The remainder of this paper is organized as follows. In Section \ref{sec:background} we introduce the required background notions. In Section \ref{sec:multi-chunk} we describe \multichunk, while in Section \ref{sec:mpfree} we describe \mpfree. In more detail, Section \ref{chunk:IO} discusses the input and output settings of \multichunk, Section \ref{chunk:Algorithm} describes the details of the algorithm, while Sections \ref{chunk:Correct}, \ref{chunk:Complexity}, and \ref{chunk:Optimal} focus on the correctness, the complexity, and the optimality of \multichunk, respectively. Similarly, Section \ref{mpfree:IO} discusses the input and output settings of \mpfree, Section \ref{mpfree:LW} describes the algorithm of Lesnick and Wright on which \mpfree\, is based, Sections \ref{mpfree:Queues} and \ref{mpfree:Minimization} focus on our improvements, while Section \ref{mpfree:Complexity} analyses the complexity of \mpfree. In Section \ref{sec:experiments} we present and discuss the experimental results. Finally, in Section \ref{sec:conclusions}, we summarize our contributions and we draw some conclusions. 
\ref{sec:app_details_1} contains proof details, 
and \ref{app:pseudocode} contains pseudocode.

\section{Definitions}
\label{sec:background}

This section presents standard definitions and facts about multi-parameter persistence modules. 
For more details, see~\cite{cz-theory, lesnick2019computing}. 
Fix a field $k$. All vector spaces in this section will be over $k$. 
Let $d \geq 1$ be a natural number. 

\begin{definition}
Let $\bbZ^d$ be the poset with 
$(w_1, \dots, w_d) \leq (z_1, \dots, z_d)$ if $w_i \leq z_i$ for all $1 \leq i \leq d$. 
A $d$-parameter \define{persistence module} is a functor 
$M : \bbZ^d \to \Vect$ from the poset $\bbZ^d$ to the category of vector spaces over $k$. 
A \define{homomorphism of persistence modules} is a natural transformation. 
For $\bfz \in \bbZ^d$ we write $M_{\bfz} = M(\bfz)$, 
and for $\textbf{w} \leq \bfz$ in $\bbZ^d$ we write 
$M_{\textbf{w}, \bfz} = M(\textbf{w} \leq \bfz) : M_{\textbf{w}} \to M_{\bfz}$.
\end{definition}

The category of persistence modules is an abelian category, 
because it is a category of functors valued in an abelian category 
\cite[8.3.6]{kashiwara-schapira}. 
So, we can form direct sums, kernels, and cokernels, and these are computed ``pointwise''. 
For example, the kernel of a homomorphism $f : M \to N$ of $d$-parameter persistence modules 
is the persistence module~$\Ker(f)$ with 
$\Ker(f)_{\bfz} = \Ker \left( f_{\bfz} : M_{\bfz} \to N_{\bfz} \right)$ 
for all $\bfz \in \bbZ^d$, 
and with the induced homomorphisms 
$\Ker(f)_{\textbf{w}, \bfz} : \Ker(f)_{\textbf{w}} \to \Ker(f)_{\bfz}$.

The term ``persistence module'' comes from the following alternative definition:

\begin{remark} \label{comm-alg-def}
Equivalently, a $d$-parameter persistence module is a module~$M$ over the polynomial ring 
$k[x_1, \dots, x_d]$ equipped with a $\bbZ^d$-grading, 
which is a vector space decomposition $M = \oplus_{\bfz \in \bbZ^d} M_{\bfz}$ 
such that $x_i \cdot M_{\bfz} \subseteq M_{\bfz + \textbf{e}_i}$ 
for all $\bfz \in \bbZ^d$, where $\textbf{e}_i$ is the $i^{\mathrm{th}}$ standard basis vector. 
An element $v$ of $M$ such that $v \in M_{\bfz}$ for some $\bfz \in \bbZ^d$ is called homogenous. 
In this language, a homomorphism $f : M \to N$ of persistence modules 
is a module homomorphism such that $f(M_{\bfz}) \subseteq N_{\bfz}$ for all $\bfz \in \bbZ^d$. 
\end{remark}

If $M : \bbZ^d \to \Vect$ is a persistence module, and $v \in M_{\bfz}$ for some $\bfz \in \bbZ^d$, 
we say that $\bfz$ is the \define{grade} of $v$, and we write $\gr(v) = \bfz$. 
In the language of \cref{comm-alg-def}, only the homogenous elements have a grade. 

A standard way to understand persistence modules is to represent them using homomorphisms 
between persistence modules with a simple form. These are the free persistence modules:

\begin{definition}
For $\textbf{u} \in \bbZ^d$, let $E^{\textbf{u}}$ be the $d$-parameter persistence module with 

\[
	E^{\textbf{u}}_{\bfz} = 
	\begin{cases}
	k \; \; \text{if} \; \textbf{u} \leq \bfz  \\
	0 \; \; \text{else}
	\end{cases}	
\text{and} \; \; \;
	E^{\textbf{u}}_{\bfw, \bfz} = 
	\begin{cases}
	\text{id} \; \; \text{if} \; \textbf{u} \leq \bfw  \\
	0 \; \; \text{else} \; .
	\end{cases}
\]
A $d$-parameter persistence module~$F$ is \define{free} 
if there is a multiset $G$ of elements in $\bbZ^d$ such that 
$F \iso \oplus_{\textbf{u} \in G} E^{\textbf{u}}$. 
Using the language of \cref{comm-alg-def}, a \define{basis} for a free persistence module~$F$ 
is a minimal homogenous set of generators for $F$. 
\end{definition}

It follows from a graded version of Nakayama's lemma 
\cite[Corollary 4.8]{eisenbud} 
that the number of elements at each grade in a basis of a free persistence module is 
an isomorphism invariant.

The input to the algorithms of this paper will be chain complexes of free persistence modules, 
which we now define. 
Because we use subscripts for grades, 
we will use superscripts for the dimension of a persistence module in a chain complex; 
note that this is contrary to the usual convention of using subscripts for the dimension. 
Because cohomology will not appear in this paper, this should not cause confusion.

\begin{definition}
A \define{chain complex} of free persistence modules $F^{\bullet}$ is a diagram of the form
\[
	 \cdots \xrightarrow{\partial^3} F^2 \xrightarrow{\partial^2} F^1 \xrightarrow{\partial^1} F^0
\]
where each $F^i$ is a free $d$-parameter persistence module, 
and $\partial^{i} \circ \partial^{i + 1} = 0$ for all $i \geq 1$. 
If $F^{\bullet}$ and $G^{\bullet}$ are chain complexes of free persistence modules, 
then a \define{chain map} $f^{\bullet} : F^{\bullet} \to G^{\bullet}$ consists of 
homomorphisms $f^i : F^i \to G^i$ for all $i \geq 0$, 
such that the following diagram commutes for all $i \geq 0$:
\[
\begin{tikzcd}
&F^{i+1} \arrow[r, "f^{i+1}"] \arrow[d, "\partial^{i+1}"'] 
&G^{i+1} \arrow[d, "\partial^{i+1}"] \\
&F^{i} \arrow[r, "f^{i}"'] 
&G^{i} \\
\end{tikzcd}
\]
\end{definition}

The homology of a chain complex of free persistence modules is defined in the usual way:

\begin{definition}
Let $F^{\bullet}$ be a chain complex of free persistence modules. For $n \geq 0$, 
the \define{homology} $H^n(F^{\bullet})$ is the $d$-parameter persistence module 
$\Ker(\partial^n) / \Im(\partial^{n+1})$. 
If $f^{\bullet} : F^{\bullet} \to G^{\bullet}$ is a chain map, 
then there are induced homomorphisms 
$f^{n}_* : H^n(F^{\bullet}) \to H^n(G^{\bullet})$. 
\end{definition}

\begin{example} \label{chain-complexes-from-simplicial-complexes}
A $d$-parameter filtered simplicial complex is a functor $S : \bbZ^d \to \SimpComplex$ 
valued in the category of simplicial complexes, 
such that for all $\textbf{w} \leq \bfz$ in $\bbZ^d$, 
the map $S_{\textbf{w}} \to S_{\bfz}$ is an inclusion. 
Let $C : \SimpComplex \to \ChainComplex$ be the oriented chain complex functor, 
so that the simplicial homology of a simplicial complex $L$ is the homology of the chain complex $C(L)$. 
Composing $S$ with $C$ we get a functor $CS : \bbZ^d \to \ChainComplex$. 
In particular, for each $i \geq 0$ we have a persistence module 
$CS^i : \bbZ^d \to \Vect$ where $CS^i_{\bfz}$ is the $i^{\mathrm{th}}$ degree part of $CS_{\bfz}$. 
If each $CS^i$ is free, then the filtered simplicial complex $S$ is called ``one-critical'', 
and we have a chain complex of free persistence modules 
$ \cdots CS^2 \to CS^1 \to CS^0$. 
Such filtered simplicial complexes are the main source of examples for this paper.
\end{example}

%\begin{example}
%Here is one way that chain complexes of free persistence modules arise in TDA. 
%Let $K$ be a finite simplicial complex, together with a function $f : K \to \bbZ^d$ 
%such that if $\sigma \in K$ and $\tau \subset \sigma$, then $f(\tau) \leq f(\sigma)$. 
%The function $f$ defines a $d$-parameter filtration of $K$, where, 
%for any $\bfz \in \bbZ^d$, we have $K_{\bfz} = \{ \sigma \in K : f(\sigma) \leq \bfz \}$. 
%For $n \geq 0$, define a free $d$-parameter persistence module~$F^{n}(K, f)$, 
%where for $\bfz \in \bbZ^d$, 
%$F^{n}(K, f)_{\bfz}$ is the $k$-vector space generated by the set of $n$-simplices $K^n_{\bfz}$. 
%Define boundary homomorphisms $\partial^{n+1} : F^{n+1}(K, f) \to F^{n}(K, f)$ using the usual formula 
%for simplicial homology. 
%The homology $H^n(F^{\bullet}(K, f))$ is the persistent homology of the $d$-parameter filtration $(K, f)$.
%\end{example}

\begin{definition}
Let $F^{\bullet}$ and $G^{\bullet}$ be chain complexes of free persistence modules. $F^{\bullet}$ and $G^{\bullet}$ are \define{isomorphic} if there exist chain maps $f^{\bullet} : F^{\bullet} \to G^{\bullet}$ and $g^{\bullet} : G^{\bullet} \to F^{\bullet}$ such that $g \circ f = \mathrm{id}_{F^{\bullet}}$ and $f \circ g = \mathrm{id}_{G^{\bullet}}$.
\end{definition}

\begin{definition}
Let $F^{\bullet}$ and $G^{\bullet}$ be chain complexes of free persistence modules. 
We say that $F^{\bullet}$ and $G^{\bullet}$ are \define{homology equivalent} 
if for all $n \geq 0$ there is an isomorphism 
$H^n(F^{\bullet}) \iso H^n(G^{\bullet})$. 
We say that $F^{\bullet}$ and $G^{\bullet}$ are \define{quasi-isomorphic} 
if there is a chain map $f : F^{\bullet} \to G^{\bullet}$ 
such that all induced homomorphisms 
$f^{n}_* : H^n(F^{\bullet}) \to H^n(G^{\bullet})$ are isomorphisms.
\end{definition}

An important class of quasi-isomorphisms are the homotopy equivalences:

\begin{definition}
Let $F^{\bullet}$ and $G^{\bullet}$ be chain complexes of free persistence modules. 
If $f,g : F^{\bullet} \to G^{\bullet}$ are chain maps, 
then we say that $f$ and $g$ are \define{chain-homotopic} if there are homomorphisms 
$h^{n} : F^{n} \to G^{n+1}$ such that 
$\partial^{n+1} \circ h^{n} + h^{n-1} \partial^{n} = f^n - g^n$ for all $n \geq 0$. 
We say that $F^{\bullet}$ and $G^{\bullet}$ are \define{homotopy equivalent} 
if there are chain maps $f : F^{\bullet} \to G^{\bullet}$ 
and $g : G^{\bullet} \to F^{\bullet}$ such that $f \circ g$ is homotopic to $\mathrm{id}_{G^{\bullet}}$ 
and $g \circ f$ is homotopic to $\mathrm{id}_{F^{\bullet}}$.
\end{definition}

If $f$ and $g$ define a homotopy equivalence between $F^{\bullet}$ and $G^{\bullet}$, 
then $f$ and $g$ are necessarily quasi-isomorphisms.

\begin{definition}
Let $M$ be a $d$-parameter persistence module. 
A \define{presentation} of $M$ is a homomorphism 
between free persistence modules $p : F^1 \to F^0$ such that $\Coker(p) \iso M$. 
Write $q : F^0 \rightarrow \Coker(p)$ for the canonical map;
the presentation $p$ is \define{minimal} if
$(1)$ $q$ maps a basis of $F^0$ bijectively to a minimal homogenous set of generators of $\Coker(p)$, and
$(2)$ $p$ maps a basis of $F^1$ bijectively to a minimal homogenous set of generators of $\Ker(q)$.
\end{definition}

We give an example of a simplicial bi-filtration and a minimal presentation of its homology 
at the end of this section. 
Any presentation of a persistence module~$M$ can be obtained (up to isomorphism) 
from a minimal presentation $p$
by taking the direct sum with maps of the form $\mathrm{id}_H : H \rightarrow H$
or of the form $H \rightarrow 0$, where $H$ is free.
Following Lesnick-Wright~\cite{lesnick2019computing}, we say that a presentation is \define{semi-minimal} 
if it can be obtained from a minimal presentation by taking the direct sum with maps of the form 
$\mathrm{id}_H : H \rightarrow H$, where $H$ is free.

\paragraph{Graded matrices}

If $F$ is a finitely-generated free $d$-parameter persistence module, 
and $B = \{B_i\}_{i \in I}$ is an ordered basis of $F$, 
then we can represent an element $v \in F_{\bfz}$ by a vector $[v]^B \in k^{|B|}$, as follows. 
If $\gr(B_i) \nleq \bfz$, we set $[v]^B_i = 0$. Otherwise, we choose $[v]^B_i$ such that we have
\[
	v = \sum_{i : \gr(B_i) \leq \bfz} [v]^B_i \cdot F_{\gr(B_i), \bfz} \left( B_i \right).
\]
Now, if $f : F \to F'$ is a homomorphism between finitely-generated free persistence modules, 
and $B, B'$ are ordered bases of $F, F'$, we can represent $f$ by a \define{graded matrix}, 
which is a matrix with entries in $k$ 
such that each column and each row is annotated with a grade in $\bbZ^d$. 
In detail, $f$ is represented by the graded matrix $[f]^{B, B'}$ whose $j^{\mathrm{th}}$ 
column is the vector $[f(B_j)]^{B'}$, 
and such that the grade of the $j^{\mathrm{th}}$ column is $\gr(B_j)$, 
and the grade of the $\ell^{\mathrm{th}}$ row is $\gr(B'_{\ell})$. 

In this paper, the input to our algorithms will be chain complexes $F^{\bullet}$ 
of finitely-generated free persistence modules, together with an ordered basis for each $F^{i}$. 
Therefore each homomorphism $\partial^{n} : F^n \to F^{n-1}$ is represented by a graded matrix. 
We assume that a graded matrix is stored in some sparse column representation, 
that is, the row indices of non-zero entries in a column, 
together with the entry in $k$, are stored in some container data structure. 
Moreover, we always assume that graded matrices have their rows and columns stored 
in an order that is compatible with the partial order on their grades, 
e.g., if the grade of column $i$ is strictly less than the grade of column $j$, then $i < j$. 

\begin{definition}
Let $A$ be a graded matrix. 
If $C$ is a non-zero column of $A$, the \define{pivot} of $C$ is the largest row index $j$ 
such that the $j^{\mathrm{th}}$ entry of $C$ is non-zero. 
We call a non-zero column \define{local} if the grade of its pivot equals the grade of the column. In such a case, we call the pivot the \define{local pivot}.
\end{definition}

We now present a small example of a minimal presentation 
that encodes the homology of a bi-filtration. 
We display the bi-filtration in \cref{fig:bifiltration}. 
We work with $\bbZ_2$ coefficients. 

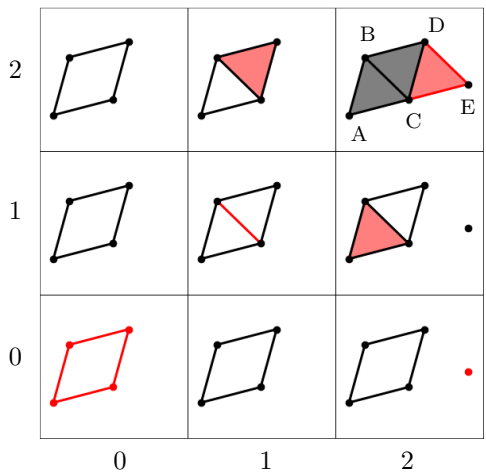
\begin{figure}[h]
\centering
\begin{overpic}[width=0.5\textwidth]{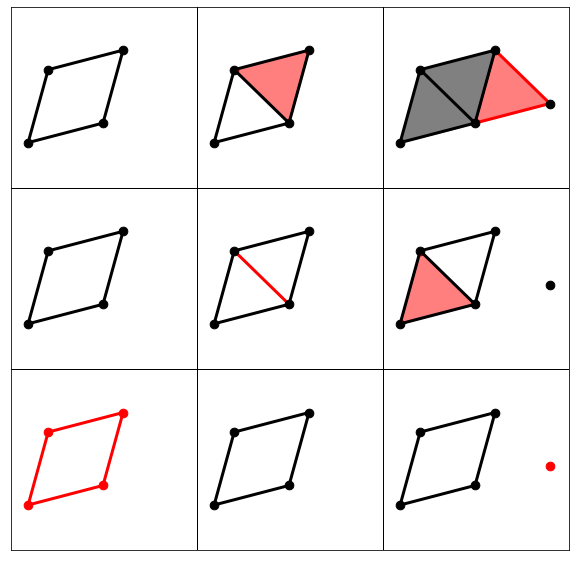}
		\put (70,67) {\footnotesize A}
		\put (72,89) {\footnotesize B}
		\put (82,70) {\footnotesize C}
		\put (87,91) {\footnotesize D}
		\put (94,73) {\footnotesize E}
		\put (18, -5) {0}
		\put (50, -5) {1}
		\put (81, -5) {2}
		\put (-5, 18) {0}
		\put (-5, 50) {1}
		\put (-5, 81) {2}
\end{overpic}
\caption{A simplicial bi-filtration. 
Simplices in red are at their minimal grade.}
\label{fig:bifiltration}
\end{figure}

To compute the homology $H^1$, we consider the boundary matrices:
\[
\scalebox{0.75}{
$[\partial^2]=\begin{array}{cc|ccc}
& & ABC & BCD & DCE\\
& & (2,1) & (1,2) & (2,2)\\
\hline
AB & (0,0) & 1 & 0 & 0\\
AC & (0,0) & 1 & 0 & 0\\
BC & (1,1) & 1 & 1 & 0\\
BD & (0,0) & 0 & 1 & 0\\
DC & (0,0) & 0 & 1 & 1\\
CE & (2,2) & 0 & 0 & 1\\
DE & (2,2) & 0 & 0 & 1
\end{array}$
}
\]
\[
\scalebox{0.75}{
$[\partial^1]=\begin{array}{cc|ccccccc}
& & AB & AC & BC & BD & DC & CE & DE\\
& & (0,0) & (0,0) & (1,1) & (0,0) & (0,0) & (2,2) & (2,2)\\
\hline
A & (0,0) & 1 & 1 & 0 & 0 & 0 & 0 & 0\\
B & (0,0) & 1 & 0 & 1 & 1 & 0 & 0 & 0\\
C & (0,0) & 0 & 1 & 1 & 0 & 1 & 1 & 0\\
D & (0,0) & 0 & 0 & 0 & 1 & 1 & 0 & 1\\
E & (2,0) & 0 & 0 & 0 & 0 & 0 & 1 & 1
\end{array}$
}
\]
\vspace{5pt}
The following graded matrix is a minimal presentation for the homology:
\[
\scalebox{0.75}{
$P = \begin{array}{c|cc}
& (1,2) & (2,1)\\
\hline
(0,0) & 1 & 0\\
(1,1) & 1 & 1
\end{array}$
}
\]
Even in this toy example we see the compression achieved by passing to a minimal presentation. 
The minimal presentation encodes the homology up to isomorphism. 
In particular, from the minimal presentation we can read off the dimension of the 
homology of the bi-filtration at each grade. 
Write $P_{\leq \bfz}$ for the sub-matrix of $P$ consisting of 
the columns and rows with grade less than or equal to $\bfz$; 
then the homology at grade $\bfz$ is isomorphic to the cokernel of $P_{\leq \bfz}$. 
For example, at $\bfz = (1,2)$, the homology has dimension equal to
\[
	\mathrm{dim} \; \mathrm{coker} \; P_{\leq (1,2)} = 
	\mathrm{dim} \; \mathrm{coker} \; \left( \begin{array}{c}
	1\\
	1
	\end{array} \right) = 1.
\]

Minimal presentations are not unique. In this case, 
the following graded matrix is also a minimal presentation:
\[
\scalebox{0.75}{
$P' = \begin{array}{c|cc}
& (1,2) & (2,1)\\
\hline
(0,0) & 0 & 1\\
(1,1) & 1 & 1
\end{array}$
}
\]
As we will discuss in \cref{sec:mpfree}, 
computing a minimal presentation involves computing a basis for 
$\mathrm{ker} \; [\partial^1]$. 
The presentations $P$ and $P'$ correspond to different bases $\mathcal{B}$ and $\mathcal{B}'$. 
Both $\mathcal{B}$ and $\mathcal{B}'$ 
include the cycle $AB + BD + DC + CA$ that is born at $(0,0)$, 
but $\mathcal{B}$ has the cycle $AB + BC + CA$ at $(1,1)$, 
while $\mathcal{B}'$ has the cycle $BC + CD + DB$ at $(1,1)$.

\section{Multi-chunk compression}
\label{sec:multi-chunk}

Given a chain complex $F^{\bullet}$ of finitely-generated free persistence modules, together with an ordered basis for each $F^{n}$, we define a compression algorithm called \multichunk \, that returns a chain complex $G^{\bullet}$ of finitely-generated free persistence modules such that:
\begin{itemize}
  \item $F^{\bullet}$ and $G^{\bullet}$ are homotopy equivalent;
%  \item for any dimension $n$ and any grade $p$, the number of $n$-generators added in $G^{\bullet}$ at $p$ coincides with the number of variations in the $n^{th}$ homology space $H^n(F^{\bullet})$ occurring at grade $p$ (see Section~\ref{chunk:Optimal} for details);
  \item given a chain complex $\bar{F}^{\bullet}$ of finitely-generated free persistence modules quasi-isomorphic to $F^{\bullet}$, for any dimension $n$ and any grade $\bfp$, the number of generators in $G^{n}$ at $\bfp$ is less than or equal to the number of generators in $\bar{F}^{n}$ at $\bfp$.
\end{itemize}

\subsection{Multi-chunk compression: Input and output settings}\label{chunk:IO}

For the sake of clarity, \multichunk \, is described for the case of 2-parameter persistence modules, and with $k=\bbZ_2$.
The algorithm can easily be reformulated and generalized for an arbitrary base field and an arbitrary number of parameters. 

Given a chain complex $F^{\bullet}$ of finitely-generated free persistence modules and a dimension $n$, choose for each $n$ a basis $B^n$ of $F^{n}$.
As in \cref{sec:background}, 
for each basis $B^n$, we fix an ordering which is compatible with the grade. That is, for two generators $g_1$ and $g_2$ of $B^n$ such that $\gr(g_1)<\gr(g_2)$, $g_1$ precedes $g_2$. 

As previously mentioned, the boundary homomorphism $\partial^{n} : F^n \to F^{n-1}$ is represented by a graded matrix $[\partial^{n}]^{B^n, B^{n-1}}$ whose
% $j^{\mathrm{th}}$ column is the vector $[\partial^{n}(B_j)]^{B'}$,
% and such that the grade of the $j^{\mathrm{th}}$ column is $\gr(B_j)$,
% and the grade of the $\ell^{\mathrm{th}}$ row is $\gr(B'_{\ell})$.
columns are in correspondence with the generators of the basis $B^n$ of $F^n$.
Specifically, each column $c$ (in correspondence with a generator $g$) encodes the boundary of $g$ with respect to the basis $B^{n-1}$ of $F^{n-1}$.
Moreover, the columns and the rows of the matrix will be sorted in accordance with the ordering adopted by $B^n$ and $B^{n-1}$. We will say that a generator $g$ of $B^n$ is of index $i$ if it is in correspondence with the $i^{th}$ column of the matrix $[\partial^{n}]^{B^n, B^{n-1}}$.
Similarly, a generator $g$ of $B^{n-1}$ is of index $i$ if it is in correspondence with the $i^{th}$ 
row of the matrix $[\partial^{n}]^{B^n, B^{n-1}}$.
Given two elements $c_1 \in F^n_{\bfz}$ and $c_2 \in F^n_{\bfw}$, we denote as $\langle c_1, c_2 \rangle$ the number (mod 2) of the generators of $B^n$ appearing with a non-zero coefficient in the representations of $c_1$ and $c_2$ in terms of the elements of $B^n$.

In the following, we will assume that the data structure adopted for encoding the matrix $[\partial^{n}]^{B^n, B^{n-1}}$ will be able to store, for each column $c$, the following information:
\begin{itemize}
    \item the index $i(c)$ of the corresponding generator $g$;
    \item the grade of $g$;
    \item the dimension of $g$;
    \item the boundary of $g$ (the list of the generators of $B^{n-1}$ which appear with a non-zero coefficient in the boundary of $g$ (as well as their indices, grades, and dimensions)).
\end{itemize}

By a small abuse of notation, in the following we will often write $c$ in place of $g$ and vice versa.

The algorithm \multichunk \, takes as input the collection of graded matrices representing the boundary homomorphisms $\partial^{n} : F^n \to F^{n-1}$ and it returns a collection of graded matrices representing boundary homomorphisms $\partial^{n} : G^n \to G^{n-1}$.

Focusing on an example, the input of \multichunk \, for the bi-filtration depicted in Figure \ref{fig:bifiltration} consists of the following boundary matrices:

\[
\scalebox{0.75}{
$[\partial^{2}]^{B^2, B^{1}}=\begin{array}{cc|ccc}
& & ABC & BCD & CDE\\
& & (2,1) & (1,2) & (2,2)\\
\hline
AB & (0,0) & 1 & 0 & 0\\
AC & (0,0) & 1 & 0 & 0\\
BD & (0,0) & 0 & 1 & 0\\
CD & (0,0) & 0 & 1 & 1\\
BC & (1,1) & 1 & 1 & 0\\
CE & (2,2) & 0 & 0 & 1\\
DE & (2,2) & 0 & 0 & {\bf 1}
\end{array}$
}
\]
\[
\scalebox{0.75}{
$[\partial^{1}]^{B^1, B^{0}}=\begin{array}{cc|ccccccc}
& & AB & AC & BD & CD & BC & CE & DE\\
& & (0,0) & (0,0) & (0,0) & (0,0) & (1,1) & (2,2) & (2,2)\\
\hline
A & (0,0) & 1 & 1 & 0 & 0 & 0 & 0 & 0\\
B & (0,0) & {\bf 1} & 0 & 1 & 0 & 1 & 0 & 0\\
C & (0,0) & 0 & {\bf 1} & 0 & 1 & 1 & 1 & 0\\
D & (0,0) & 0 & 0 & {\bf 1} & {\bf 1} & 0 & 0 & 1\\
E & (2,0) & 0 & 0 & 0 & 0 & 0 & 1 & 1
\end{array}$
}
\]
expressed in terms of the following bases (already sorted with an ordering compatible with the grades):
$B^0=\{A, B, C, D, E\}$, 
$B^1=\{AB$, $AC$, $BD$, $CD$, $BC$, $CE$, $DE\}$, and 
$B^2=\{ABC, BCD, CDE\}$. 
%\begin{itemize}
%    \item $B^0=\{A, B, C, D, E\}$;
%    \item $B^1=\{AB, AC, BD, CD, BC, CE, DE\}$;
%    \item $B^2=\{ABC, BCD, CDE\}$.
%\end{itemize}
In order to highlight them, the local pivots of the boundary matrices (which will play a crucial role in the first phase of \multichunk) have been written in bold.

%The output of \multichunk \, will be a collection of graded matrices $[\partial^{n}]^{C^n, C^{n-1}}$ (encoded adopting the same data structure) representing the boundary homomorphisms $\partial^{n} : H^n \to H^{n-1}$ expressed in terms of the bases $C^n$ and $C^{n-1}$ between the two finitely-generated free persistence modules $H^n$ and $H^{n-1}$ of dimension $n$ and $n-1$, respectively.

\subsection{Multi-chunk compression: Algorithm}\label{chunk:Algorithm}

The algorithm \multichunk \, works in three phases:
\begin{itemize}
  \item {\em local reduction};
  \item {\em compression};
  \item {\em removal of local columns}.
\end{itemize}

The next paragraphs will be devoted to describing these phases. Pseudocode of the three phases is provided in \ref{app:pseudocode}.
% Each of the phases will consists of algorithms to be applied to all the graded matrices returned as output by the previous step.

\paragraph{Phase I: Local reduction}
The goal of this phase is to perform a preliminary reduction of the graded matrices aimed at labelling the generators of each basis $B^n$ as \emph{global}, \emph{local positive} or \emph{local negative}.

% A crucial notion for this phase is the one of local pivot. 
% Given a column $c$, we define the {\em local pivot} of $c$ as the generator $c'$ in the boundary of $c$ with maximal index such that $\gr(c')=\gr(c)$. If no such generator exists, we simply say that $c$ has no local pivot.

We proceed in decreasing order with respect to the dimension $n$. 
The phase of local reduction works as follows on the graded boundary matrix $[\partial^{n}]^{B^n, B^{n-1}}$ (see Alg.~\ref{alg:phaseI}).  
% At the beginning all columns of the matrix are initialized as unlabeled.
The algorithm traverses the columns in increasing order with respect to the index $i$ and performs the following operations on a column $c$.
If $c$ (or better, its corresponding generator $g$) has been already labeled (as global, local positive, or local negative), do nothing.
Otherwise,
as long as $c$ has a local pivot (i.e., if the grade of the pivot of $c$ equals the one of $c$) and there is a column $c'$ with $i(c')<i(c)$
and the same local pivot as $c$,
perform the column addition $c\gets c+ c'$.
% where $\lambda$ is chosen such that the local pivot of $c$ disappears.
If at the end of this loop, the column $c$ does not have a local pivot, label the column as global and proceed.
Otherwise, label $c$ as local negative and its local pivot $c''$ as local positive.
In this case, $(c'',c)$ is called a \emph{local pair}.

Any column addition performed within the local reduction involves columns having the same grade. So, the local reduction operates independently on columns
of different grade. We call blocks of columns with the same grade \emph{chunks}; hence the name of the algorithm. As just discussed, operations on one chunk do not affect columns of any other chunk, hence the local reduction phase can be readily invoked in parallel on the chunks of the input matrix $[\partial^{n}]^{B^n, B^{n-1}}$.

% Worth to be noticed that, choosing a specific $n$ and adopting on the corresponding graded boundary matrix the above described strategy of local reduction is not enough to label each generator of the bases $B^n$ and $B^{n-1}$. 
A complete labeling of all the generators of the chosen bases of $F^\bullet$ is achieved by applying local reduction to all the graded boundary matrices of $F^\bullet$.
Finally, note that by proceeding in decreasing dimension, we avoid performing any column additions on local positive columns. This is reminiscent of the \emph{clearing optimization} in the one-parameter version~\cite{ck-twist,chunk}.

Running Phase I of \multichunk \, on the considered example, we obtain the following boundary matrices:
\[
\scalebox{0.75}{
$[\partial^{2}]^{B^2, B^{1}}=\begin{array}{cc|ccc}
& & ABC & BCD & \textcolor{red}{CDE}\\
& & (2,1) & (1,2) & (2,2)\\
\hline
\textcolor{red}{AB} & (0,0) & {\bf 1} & 0 & 0\\
\textcolor{red}{AC} & (0,0) & {\bf 1} & 0 & 0\\
\textcolor{red}{BD} & (0,0) & 0 & {\bf 1} & 0\\
CD & (0,0) & 0 & 1 & 1\\
BC & (1,1) & 1 & 1 & 0\\
CE & (2,2) & 0 & 0 & 1\\
\textcolor{green}{DE} & (2,2) & 0 & 0 & 1
\end{array}$
}
\]
\[
\scalebox{0.75}{
$[\partial^{1}]^{B^1, B^{0}}=\begin{array}{cc|ccccccc}
& & \textcolor{red}{AB} & \textcolor{red}{AC} & \textcolor{red}{BD} & CD & BC & CE & \textcolor{green}{DE}\\
& & (0,0) & (0,0) & (0,0) & (0,0) & (1,1) & (2,2) & (2,2)\\
\hline
A & (0,0) & 1 & 1 & 0 & 0 & 0 & 0 & 0\\
\textcolor{green}{B} & (0,0) & 1 & 0 & 1 & 0 & {\bf 1} & 0 & 0\\
\textcolor{green}{C} & (0,0) & 0 & 1 & 0 & 0 & {\bf 1} & {\bf 1} & 0\\
\textcolor{green}{D} & (0,0) & 0 & 0 & 1 & 0 & 0 & 0 & 1\\
E & (2,0) & 0 & 0 & 0 & 0 & 0 & 1 & 1
\end{array}$
}
\]
and the following local pairs: $(DE, CDE)$, $(B, AB)$, $(C, AC)$, $(D, BD)$. 
Columns and rows highlighted in red will correspond to local negative generators, while the green ones to local positive generators. 
The remaining columns and rows will correspond to global generators. 
Since they will play a crucial role in the Phase II of \multichunk, local negative and local positive elements currently appearing in the boundary of global columns have been highlighted in bold.

\paragraph{Phase II: Compression}
The second phase of the algorithm aims at removing local (positive or negative)
generators from the boundary of all global columns in the graded boundary matrices. 

Given the modified graded boundary matrix $[\partial^{n}]^{B^n, B^{n-1}}$  returned by the previous step, the compression phase works as follows (see Alg.~\ref{alg:phaseII}).
For each global column $c$, while the boundary of the column contains a generator that is local positive or local negative, the algorithm picks the local generator $c'$ with maximal index.
\begin{itemize}
  \item If $c'$ is negative, remove $c'$ from the boundary of $c$;
  \item If $c'$ is positive, denote $c''$ as the (unique) local negative $n$-column with $c'$ as local pivot and perform the column addition $c\gets c+ c''$.
%   where $\lambda$ is chosen such that $c'$ disappears in the boundary of $c$.
\end{itemize}
This ends the description of the compression phase.
On termination, all the generators in the boundary of a global $n$-column are global.

The above process terminates for a column $c$ because the index of the maximal
local generator in the boundary of $c$ is strictly decreasing in each step.
That is clear for the case that $c'$ is local negative. If $c'$ is local positive,
then $c'$ is the generator in the boundary of $c''$ with the maximal index, so the column
addition does not introduce in the boundary of $c$ any generators with a larger index.

Note that the compression of a global column does not affect the result
on any other global column. Thus, the phase can be parallelized as well.

Running Phase II of \multichunk \, on the considered example, we obtain the following boundary matrices:
\[
\scalebox{0.75}{
$[\partial^{2}]^{B^2, B^{1}}=\begin{array}{cc|ccc}
& & ABC & BCD & \textcolor{red}{CDE}\\
& & (2,1) & (1,2) & (2,2)\\
\hline
\textcolor{red}{AB} & (0,0) & 0 & 0 & 0\\
\textcolor{red}{AC} & (0,0) & 0 & 0 & 0\\
\textcolor{red}{BD} & (0,0) & 0 & 0 & 0\\
CD & (0,0) & 0 & 1 & 1\\
BC & (1,1) & 1 & 1 & 0\\
CE & (2,2) & 0 & 0 & 1\\
\textcolor{green}{DE} & (2,2) & 0 & 0 & 1
\end{array}$
}
\]
\[
\scalebox{0.75}{
$[\partial^{1}]^{B^1, B^{0}}=\begin{array}{cc|ccccccc}
& & \textcolor{red}{AB} & \textcolor{red}{AC} & \textcolor{red}{BD} & CD & BC & CE & \textcolor{green}{DE}\\
& & (0,0) & (0,0) & (0,0) & (0,0) & (1,1) & (2,2) & (2,2)\\
\hline
A & (0,0) & 1 & 1 & 0 & 0 & 0 & 1 & 0\\
\textcolor{green}{B} & (0,0) & 1 & 0 & 1 & 0 & 0 & 0 & 0\\
\textcolor{green}{C} & (0,0) & 0 & 1 & 0 & 0 & 0 & 0 & 0\\
\textcolor{green}{D} & (0,0) & 0 & 0 & 1 & 0 & 0 & 0 & 1\\
E & (2,0) & 0 & 0 & 0 & 0 & 0 & 1 & 1
\end{array}$
}
\]

\paragraph{Phase III: Removal of local pairs}
%
%The third phase of \multichunk,  aims at removing a large number of generators of the finitely-generated free persistence modules of the input chain complex $F^{\bullet}$. 
In the third phase of \multichunk, we remove columns and rows from the boundary matrices.

Given the modified graded boundary matrices returned by the previous step, Phase III works as follows (see Alg.~\ref{alg:phaseIII}).
Traverse all columns, and remove all columns labeled as local (positive or negative). Similarly, traverse all rows, and remove all rows labeled as local (positive or negative).

Let us denote by $C^n$ the collection of generators of $B^n$ in correspondence with the remaining (global) columns. 
The reduced matrices obtained by performing Phase III for all dimensions $n$ can be considered as the graded boundary matrices $[\partial^{n}]^{C^n, C^{n-1}}$ representing the boundary homomorphisms $\partial^{n} : G^n \to G^{n-1}$ expressed in terms of the bases $C^n$ and $C^{n-1}$.

Running Phase III of \multichunk \, on the considered example, we obtain as output the following boundary matrices:
\[
\scalebox{0.75}{
$[\partial^{2}]^{C^2, C^{1}}=\begin{array}{cc|cc}
& & ABC & BCD \\
& & (2,1) & (1,2) \\
\hline
CD & (0,0) & 0 & 1 \\
BC & (1,1) & 1 & 1 \\
CE & (2,2) & 0 & 0
\end{array}$
}
\]
\[
\scalebox{0.75}{
$[\partial^{1}]^{C^1, C^{0}}=\begin{array}{cc|ccc}
& & CD & BC & CE \\
& & (0,0) & (1,1) & (2,2) \\
\hline
A & (0,0) & 0 & 0 & 1 \\
E & (2,0) & 0 & 0 & 1 
\end{array}$
}
\]
expressed in terms of the following bases: 
$C^0=\{A, E\}$, $C^1=\{CD$, $BC$, $CE\}$, and 
$C^2=\{ABC, BCD\}$.
%\begin{itemize}
%    \item $C^0=\{A, E\}$;
%    \item $C^1=\{CD, BC, CE\}$;
%    \item $C^2=\{ABC, BCD\}$.
%\end{itemize}

% \textcolor{red}{ICO: Do you think that a ``geometric realization'' of the obtained output could help the reader? Please let me known.}

\subsection{Multi-chunk compression: Correctness}\label{chunk:Correct}

We prove that the output of \multichunk \, is homotopy equivalent to the input. 
In order to do that, 
\begin{itemize}
    \item we introduce two elementary operations on chain complexes of finitely-generated free persistence modules;
    \item we prove that both elementary operations do not modify the homotopy type of the chain complex;
    \item we show that \multichunk \, can be expressed by a sequence of such elementary operations.
\end{itemize}

\paragraph{Grade-preserving column addition}
Given distinct columns $c, c'$ of a graded boundary matrix $[\partial^{n}]^{B^n, B^{n-1}}$, an operation of the form $c\gets c+ c'$ is called \emph{grade-preserving} if $\gr(c')\leq \gr(c)$. Note that such an operation maintains the property that any generator $c''$ in the boundary of $c$ satisfies $\gr(c'')\leq \gr(c)$, by transitivity of the grade.

\paragraph{Removal of local pair}
We call $(c_1,c_2)$ a \emph{local pair} if $c_1$ is a generator in $B^n$, $c_2$ is a generator in $B^{n+1}$, $\gr(c_1)=\gr(c_2)$ and $c_1$ is the local pivot of $c_2$.
We call {\em removal of the local pair} $(c_1, c_2)$ the operation $Del(c_{1},c_{2})$ which acts on the graded boundary matrices of a chain complex $F^{\bullet}$ as follows.
\begin{itemize}
\item For every $(n+1)$-column $c$, replace its boundary $\partial^{n+1}(c)$ with $\partial^{n+1}(c) + \mu \partial^{n+1}(c_2)$, where $\mu$ is the coefficient of $c_1$ in $\partial^{n+1}(c)$.
% are the coefficients of $c_1$ in $\partial^{n+1}(c_2)$ and in $\partial^{n+1}(c)$, respectively. 
In particular, after this operation, $c_1$ disappeared from the boundary of any $(n+1)$-columns. Moreover, notice that this operation is a grade-preserving column addition because the pair $(c_1,c_2)$ is local,
that is, $\gr(c_1)=\gr(c_2)$.
% except $c_2$.
\item For every $(n+2)$-column $c$, update its boundary by setting the coefficient of $c_{2}$ in $\partial^{n+2}(c)$ to $0$. In terms of matrices, this corresponds to removing the row corresponding to $c_2$ from $[\partial^{n}]^{B^{n+1}, B^{n}}$.
\item Delete the generators $c_1$ and $c_{2}$ from $B^{n}$ and $B^{n+1}$, respectively.
\end{itemize}

\paragraph{Grade-preserving column addition preserves homotopy type}
Given two $n$-columns $c_1, c_2$ such that $\gr(c_1)\leq \gr(c_2)$
% and a scalar value $\lambda \in k$
, the algorithm we propose allows for adding 
% $\lambda$ copies of 
the boundary of column $c_1$ to the boundary of column $c_2$. In this section, we formalize that in terms of modifications of chain complexes.

Given a chain complex $F^\bullet=(F^l, \partial^l)$ of finitely generated, 
free persistence modules, and distinct generators $c_1, c_2$ in the basis $B^n$ 
with $\gr(c_1)\leq \gr(c_2)$, we define 
$\bar{F}^\bullet=(\bar{F}^l, \bar{\partial}^l)$ by setting:
%Given a chain complex $F^\bullet=(F^l, \partial^l)$ endowed with a grade function $\gr:F^\bullet\rightarrow \bbZ^2$, let us consider two generators $c_1, c_2$ among the ones of the space of the $n$-chains $F^n$ such that $\langle c_1, c_2 \rangle=0$ and $\gr(c_1)\leq \gr(c_2)$. 
%% Chosen a scalar value $\lambda \in k$, 
%Let us define $\bar{F}^\bullet=(\bar{F}^l, \bar{\partial}^l)$ by setting:
\begin{itemize}
  \item $\bar{F}^l=F^l$,
  \item for any $c\in B^l$, \[\bar{\partial}^l(c)=\begin{cases}
  \partial^n(c) +  \langle c, c_2 \rangle \partial^n (c_1) & \text{ if } l=n,\\
  \partial^{n+1}(c) + \langle \partial^{n+1} (c), c_2 \rangle c_1 & \text{ if } l=n+1,\\
  \partial^l(c) & \text{ otherwise.}
  \end{cases}\]
\end{itemize}
Let us define chain maps $f^\bullet: F^\bullet \rightarrow \bar{F}^\bullet$, $g^\bullet: \bar{F}^\bullet \rightarrow F^\bullet$ as follows:
\begin{itemize}
  \item for any $c\in B^l$, \[f^l(c)=\begin{cases}
  c + \langle c, c_2 \rangle c_1 & \text{ if } l=n,\\
  c & \text{ otherwise;}
  \end{cases}\]
  \item for any $c\in B^l$, \[g^l(c)=\begin{cases}
  c +  \langle c, c_2 \rangle c_1 & \text{ if } l=n,\\
  c & \text{ otherwise.}
  \end{cases}\]
\end{itemize}
Then, by a routine calculation we obtain the following:

\begin{proposition}
The chain complexes of free persistence modules $F^\bullet$ and $\bar{F}^\bullet$ are isomorphic via the chain map $f^\bullet$ and its inverse $g^\bullet$.
\end{proposition}

\paragraph{Removal of local pair preserves homotopy type}
Given a local pair $(c_1, c_2)$ of columns, the algorithm we propose allows for deleting them from the boundary matrix. In this section, we formalize that in terms of modifications of chain complexes.

Given a chain complex $F^\bullet=(F^l, \partial^l)$ of finitely generated, 
free persistence modules, and generators $c_1 \in B^n$ and $c_2 \in B^{n+1}$ 
such that $\langle \partial^{n+1}(c_2), c_1 \rangle \neq 0$ and $\gr(c_1)=\gr(c_2)$, 
we define 
$\bar{F}^\bullet=(\bar{F}^l, \bar{\partial}^l)$ by setting:
\begin{itemize}
  \item $\bar{F}^l$ as the free persistence module with basis \[\bar{B}^l=\begin{cases}
  B^n - \{c_1\} & \text{ if } l=n,\\
  B^{n+1} - \{c_2\} & \text{ if } l=n+1,\\
  B^l & \text{ otherwise;}
  \end{cases}\]
  \item for any $c\in \bar{B}^l$, \[\bar{\partial}^l(c)=\begin{cases}
  \partial^{n+1}(c) + \langle \partial^{n+1}(c), c_1 \rangle \partial^{n+1}(c_2) & \text{ if } l=n+1,\\
  \partial^{n+2}(c) + \langle \partial^{n+2} (c), c_2 \rangle c_2 & \text{ if } l=n+2,\\
  \partial^l(c) & \text{ otherwise.}
  \end{cases}\]
\end{itemize}
Let us define chain maps $r^\bullet: F^\bullet \rightarrow \bar{F}^\bullet$, $s^\bullet: \bar{F}^\bullet \rightarrow F^\bullet$ as follows:
\begin{itemize}
  \item for any $c\in B^l$, \[r^l(c)=\begin{cases}
  c + \langle c, c_1 \rangle \partial^{n+1}(c_2) & \text{ if } l=n,\\
  c + \langle c, c_2 \rangle c_2 & \text{ if } l=n+1,\\
  c & \text{ otherwise;}
  \end{cases}\]
  \item for any $c\in \bar{B}^l$, \[s^l(c)=\begin{cases}
  c + \langle \partial^{n+1}(c), c_1 \rangle c_2 & \text{ if } l=n+1,\\
  c & \text{ otherwise;}
  \end{cases}\]
\end{itemize}
Using these maps, we prove the following result (see~\ref{sec:app_details_1} for a detailed proof):

\begin{theorem} \label{local-pair-homotopy-equivalence}
The chain complexes of free persistence modules $F$ and $\bar{F}$ are homotopy equivalent.
\end{theorem}

\paragraph{Multi-chunk algorithm preserves homotopy type}

By combining the previous results, we are now ready to prove the correctness of \multichunk.  

\begin{theorem}
Let $F^{\bullet}$ be a chain complex of finitely-generated free persistence modules together with an ordered basis for each $F^{n}$ and let $G^{\bullet}$ be the chain complex of finitely-generated free persistence modules returned by performing 
\textup{\texttt{multi-chunk}} on $F^{\bullet}$. 
Then, $F^{\bullet}$ and $G^{\bullet}$ are homotopy equivalent. 
\end{theorem}

\begin{proof}
In order to prove the theorem we exhibit that \multichunk \,  can be expressed by a sequence of grade-preserving column additions and removals of local pairs.
Because every column addition in Phase I is between columns of the same
grade, all column additions are grade-preserving.
Hence, after Phase I, the chain complex is equivalent to the input.

In Phase II, note that all column additions performed are grade-preserving.
Indeed, if $c'$ is in the boundary of column $c$, then $\gr(c')\leq \gr(c)$ holds.
If $c'$ is local positive, it triggers a column addition
of the form $c\gets c+ c''$ with
its local negative counterpart $c''$. Since $\gr(c')=\gr(c'')$,
$\gr(c'')\leq \gr(c)$ as well.

A further manipulation in Phase II is the removal of local negative columns from the boundary of global columns.
These removals cannot be directly expressed in terms of the two elementary operations
from above. Instead, we define a slight variation of our algorithm:
in Phase II, when we encounter a local negative $c'$, we do nothing.
In other words, the compression only removes the local positive generators
from the boundary $c$, and keeps local negative and global generators.
In Phase III, instead of removing local columns, we perform a removal
of a local pair $(c_1,c_2)$ whenever we encounter a local negative column $c_2$
with local pivot $c_1$. We call that algorithm \texttt{modified multi-chunk}.
Note that this modified algorithm is a sequence of grade-preserving
column additions, followed by a sequence of local pair removals, and thus
produces a chain complex that is equivalent to the input $F^{\bullet}$.

We argue next that \multichunk \,  and \texttt{modified multi-chunk}
yield the same output. Since both versions eventually remove all local
columns, it suffices to show that they yield the same global columns.
Fix an index of a global column, and let $c$ denote the column of that index
returned by the original chunk algorithm. Let $c^\ast$ denote the
column of the same index produced by the modified algorithm after
the modified Phase II. The difference of $c^\ast$ and $c$ lies in the
presence of local negative generators in the boundary of $c^\ast$ which have been removed in $c$.
The modified Phase III affects $c^\ast$ in the following way:
when a local pair $(c_1,c_2)$ is removed, the local negative
$c_2$ is, if it is present, removed from the boundary of $c^\ast$. There is no column addition
during the modified Phase III involving $c^\ast$ because all local positive
columns have been eliminated. Hence, the effect of the modified Phase III
on $c^\ast$ is that all local negative columns are removed from its boundary which turns
$c^\ast$ to be equal to $c$ at the end of the algorithm.
Hence, the output of both algorithms is the same, proving the theorem.
\end{proof}

\subsection{Multi-chunk compression: Complexity}\label{chunk:Complexity}

In order to properly express the time and the space complexity of the proposed algorithm, let us introduce the following parameters. Given input $F^{\bullet}$, we denote as $n$ the number of generators of $F^{\bullet}$, as $m$ the number of chunks, as $\ell$ the maximal size of a chunk, and as $g$ the number of global columns. Moreover, we assume the maximal size of the support of the boundary of the generators of $F^{\bullet}$ as a constant. The latter condition is always ensured for bi-filtered simplicial complexes of fixed dimension.

\begin{theorem}
\label{thm:multi-chunk-complexity}
The algorithm \textup{\multichunk} has time complexity $O(m\ell^3+g\ell n)$ and space complexity $O(n\ell + g^2)$.
\end{theorem}
\begin{proof}
Due to the similarity between the two algorithms, the analysis of complexity of the proposed algorithm is analogous to the first two steps  of the one-parameter chunk algorithm~\cite{chunk}.
% The additional factor of $\log\ell$ comes from our choice of using priority queues as column type and could be removed by using list representations as in~\cite{chunk}.%
% \footnote{We remark, however, that this would result in a performance penalty in practice. See~\cite{Bauer2017}.}

On the space complexity, during the Phase I, the generators in the boundary of any column can be at most $O(\ell)$. So, $O(n\ell)$ is a bound on the accumulated size of all columns
after Phase I.
During Phase II, the boundary of any global column can consist of up to $n$ generators, but reduces to $g$ generators at the end of the compression of the column
because all local entries have been removed. Hence, the final chain complex has at most $g$ entries in each of its $g$ columns, and requires $O(g^2)$ space.
Hence, the total space complexity is $O(n\ell+n+g^2)$, where the second summand is redundant.\footnote{We remark that this bound only holds for the sequential version of the
algorithm. In a parallelized version, it can happen that several compressed columns achieve a size of $O(n)$ at the same time.}
\end{proof}

We remark that the first term in the bound of Theorem~\ref{thm:multi-chunk-complexity} can actually be improved to $\min\{m \ell^3,n^3\}$ because in the worst
case, the local reduction procedure (whose complexity is counted by this term)
cannot perform more than $O(n^2)$ column operations, each of them costing
at most $O(n)$ time.

\subsection{Multi-chunk compression: Optimality}\label{chunk:Optimal}

Let $\bar{F}^{\bullet}$ be a chain complex of finitely generated free persistence modules. 
Let $\bfp = (p_x, p_y) \in \bbZ^2$, and define the vector space $\bar{F}^n_{<\bfp}$ 
to be the colimit of the restriction 
$\bar{F}^n |_{\bbZ^2_{<\bfp}} : \bbZ^2_{<\bfp} \to \Vect$, 
where $\bbZ^2_{<\bfp} = \{ \bfz \in \bbZ^2 : \bfz < \bfp \}$. 
Using the induced boundary homomorphisms, we obtain a chain complex $\bar{F}^{\bullet}_{<\bfp}$ 
of vector spaces. 
Equivalently, if we write 
$\bfp_0 = (p_x - 1, p_y - 1), 
\bfp_1 = (p_x - 1, p_y)$, and 
$\bfp_2 = (p_x, p_y - 1)$, 
the vector space $\bar{F}^n_{<\bfp}$ can be defined by the short exact sequence:
\[
0 \to \bar{F}_{\bfp_0}^n \to 
\bar{F}_{\bfp_1}^n \oplus \bar{F}_{\bfp_2}^n \to 
\bar{F}_{<\bfp}^n \to 0 \; .
\]
Moreover, let $\eta_{\bfp}^n$ be the homology map in dimension $n$ induced by the inclusion of $\bar{F}_{<\bfp}^\bullet$ into $\bar{F}_{\bfp}^\bullet$.
We define
\[\delta_{\bfp}^n(\bar{F}^{\bullet}):=\dim\,\Ker \, \eta_{\bfp}^{n-1} + \dim\,\Coker \, \eta_{\bfp}^n,\]
and
\[\gamma_{\bfp}^n(\bar{F}^{\bullet}):=\dim\,\bar{F}_{\bfp}^n - \dim\,\bar{F}_{<\bfp}^n.\]

\begin{theorem}\label{thm:optimality}
Let $G^{\bullet}$ be the output of \textup{\multichunk} applied to the input $F^{\bullet}$. Then $\delta_{\bfp}^n(F^{\bullet})=\gamma_{\bfp}^n(G^{\bullet})$.
\end{theorem}

The full proof is given in~\ref{sec:app_details_1}
and can be summarized as
follows. Every global column with grade $\bfp$ either destroys a homology class of $H(F^{\bullet}_{<\bfp})$,
or it creates a new homology class in $H(F^{\bullet}_{\bfp})$, which is not destroyed by any other column of grade $\bfp$. Hence, each global column contributes a generator to $\Ker \, \iota_{\bfp}^{n-1}$ or to $\Coker \, \iota_{\bfp}^n$, where $\iota_{\bfp}^l$ is the map between the $l^{th}$ homology spaces induced by the inclusion of $F_{<\bfp}^\bullet$ into $F_{\bfp}^\bullet$.
Local columns do not contribute to either of these two spaces.
The result follows from the fact that the number of global columns
at grade $\bfp$ is precisely the number of generators added at $\bfp$ to $G^{\bullet}$.

The next statement shows that our construction is optimal
in the sense that any chain complex of free persistence modules $\bar{F}^{\bullet}$ that is quasi-isomorphic to $F^{\bullet}$ must have at least as many generators as $G^{\bullet}$.

\begin{theorem}\label{lemma:leq}
Any chain complex of free persistence modules $\bar{F}^{\bullet}$ quasi-isomorphic to $F^{\bullet}$ has to add at least $\delta_{\bfp}^n(F^{\bullet})$ $n$-generators at grade $\bfp$. I.e., $\delta_{\bfp}^n(F^{\bullet})\leq \gamma_{\bfp}^n(\bar{F}^{\bullet})$.
\end{theorem}

The full proof is also given in~\ref{sec:app_details_1}.
To summarize,
it is not too hard to see that, for any $\bar{F}^{\bullet}$ 
quasi-isomorphic to $F^{\bullet}$,
$\delta_{\bfp}^n(\bar{F}^{\bullet})\leq \gamma_{\bfp}^n(\bar{F}^{\bullet})$. 
Moreover, the quasi-isomorphism between $F^{\bullet}$ and $\bar{F}^{\bullet}$ 
implies that $\dim\,\Ker \, \eta_{\bfp}^{n-1}=\dim\,\Ker \, \iota_{\bfp}^{n-1}$
and $\dim\,\Coker \, \eta_{\bfp}^n=\dim\,\Coker \, \iota_{\bfp}^n$. So,
$\delta_{\bfp}^n(F^{\bullet})= \delta_{\bfp}^n(\bar{F}^{\bullet})$, 
which implies that claim. The equality of the dimensions
is formally verified using the Mayer-Vietoris sequence and the $5$-lemma
to establish an isomorphism from $H^n(F_{<\bfp}^\bullet)$ to $H^n(\bar{F}_{<\bfp}^\bullet)$
that commutes with the isomorphism at grade $\bfp$.

% \subsection{Multi-chunk compression: Detailed proofs}
% \textcolor{red}{Should this subsection be moved into the Appendix? I think so.}

% \subsection{Multi-chunk compression: Implementation and experimental results}

\section{Mpfree compression}
\label{sec:mpfree}

We now describe \mpfree, our algorithm for computing minimal presentations 
of $2$-parameter persistence modules. 
As with the description of \multichunk, we fix the base field $k = \bbZ_2$ for simplicity, 
but the algorithm can easily be reformulated for an arbitrary base field. 
We assume the number of parameters $d$ is $2$. 
For \mpfree \, this assumption is not simply for convenience. 
For $d > 2$, the methods of this section do not immediately apply.

\subsection{Mpfree compression: Input and output settings}\label{mpfree:IO}

In \cref{sec:multi-chunk} we assume that bases of free persistence modules are ordered 
in a way that is compatible with the partial order on the grades of the basis elements. 
We now strengthen this assumption, and assume that bases of free persistence modules 
are in co-lexicographic order with respect to grades, 
with basis elements of the same grade ordered arbitrarily. 
In particular, the graded matrices in this section have columns and rows that are 
in co-lexicographic order with respect to grades. 
We explain the significance of this assumption when we describe 
the column operations performed by the algorithm.

As explained in the introduction, \mpfree \, is a modified version of an algorithm of Lesnick and Wright, 
which we call the LW algorithm. 
The LW algorithm takes as input a short chain complex 
\[
	F^2 \xrightarrow{\partial^2} F^1 \xrightarrow{\partial^1} F^0
\]
of finitely generated free $2$-parameter persistence modules, 
and returns a minimal presentation of the homology of this chain complex. 
More precisely, the input consists of graded matrices $A$ and $B$, 
representing $\partial^2$ and $\partial^1$ respectively.

The LW algorithm consists of four steps, where the first three compute a semi-minimal presentation $M'$, 
and the fourth step computes a minimal presentation $M$ from $M'$; 
%We begin by describing the LW algorithm on a high level.
our algorithm \mpfree \, shares this structure.

\begin{description}
\item [\texttt{Min\_gens}]
Computes a minimal ordered set of generators $G$ of the image of $A$.
\item [\texttt{Ker\_basis}] Computes an ordered basis $K$ of the kernel of $B$.
\item [\texttt{Reparam}] 
Re-expresses every element of $G$
as a linear combination in $K$, keeping its grade. 
This is possible since $BA=0$.
The resulting
graded matrix $M'$ is a semi-minimal presentation.
\item [\texttt{Minimize}] Identifies pairs $(g,r)$ of generators (rows in $M'$)
and relations (columns in $M'$)
where $r$ ``eliminates'' $g$ and both $r$ and $g$ have the same grade. 
In that case, row $g$ and column $r$ are removed from $M'$ after some
algebraic manipulations without changing the persistence module.
Removing all such pairs results in a minimal presentation $M$.
\end{description}

%The algorithm \mpfree \, also follows this strategy; 
The difference between the LW algorithm and \mpfree \, is 
in how these steps are implemented. 
We will now describe the LW algorithm, 
and then describe our improvements. 
We give pseudocode in \ref{app:pseudocode}.

\subsection{Mpfree compression: The Lesnick--Wright algorithm}\label{mpfree:LW}

As we assume $d=2$, each column and row of the graded matrices $A$ and $B$ 
come with a grade in $\bbZ^2$. 
We assume that every $x$-coordinate of a grade of a column of $A$ or $B$ is in $\{1,\ldots,X\}$ 
and every $y$-coordinate in $\{1,\ldots,Y\}$. 
Hence, we can visualize $A$ and $B$ via a $X \times Y$ integer grid, where each grid cell contains
a (possibly empty) sequence of matrix columns. Traversing the grid
row by row upwards yields the co-lexicographic order of the matrix. 
We will phrase the algorithms for
\texttt{Min\_gens} and \texttt{Ker\_basis} with this interpretation. 
To illustrate the algorithm we show how each step acts on our running example, 
displayed in \cref{fig:bifiltration}.

\paragraph{Details of \textup{\texttt{Min\_gens}}}
The procedure traverses the columns of $A$ in a certain order defined below.
During the traversal, it maintains a \define{pivot map} $\rho$, 
a partial map from row indices to column indices. The interpretation is that
$\rho(i)=j$ if column $j$ has been visited, has pivot $i$, and there is
no visited column $j'$ with pivot $i$ and $j'<j$. 
Initially, $\rho$ is the empty map, reflecting the state that no column
has been visited.

At any point of the algorithm, \define{reducing} a column $j$ means
the following operation (Alg.~\ref{alg:reduce_lw}): as long as $j$ is not empty, has pivot $i$
and $\rho(i)=j'$ with $j'<j$, add column $j'$ to column $j$.
This results in cancellation of the pivot (since the coefficients are in $\bbZ_2$)
and hence, after the addition,
the pivot of $j$ is strictly smaller than $i$ (or the column is empty).
In either case, the reduction terminates after finitely many iterations, 
and column $j$ is marked as visited. If it ends with a non-empty
column with pivot $i$, set $\rho(i)\gets j$.

We can now describe the procedure \texttt{Min\_gens} (see Alg.~\ref{alg:min_gens_lw}):
Using the grid interpretation from above, traverse the grid cells
in lexicographic order; 
that means, the grid is traversed
column by column from the left, traversing each column bottom-up.
When reaching grid cell $(x,y)$, iterate through all
matrix columns with grade $(1,y), (2,y),\ldots, (x-1,y)$ in that order
(i.e., through all cells on the left of $(x,y)$)
and reduce them as described above. 
Then, iterate through the matrix
columns at grade $(x,y)$ and reduce them as well. 
Append every column at grade $(x,y)$ 
not reducing to $0$ in the output matrix, with grade $(x,y)$.
%(which is the grade of the currently visited column in $A$).
%
%Note that because the columns of $A$ are stored in co-lexicographic order,
%the columns with grades $(1,y),\ldots,(x,y)$ form one consecutive block
%of columns which can be determined efficiently. In fact, by precomputing
%one index per grid cell, 
%%(in linear time with respect to the number of columns
%%and the number of grid cells), 
%the index range for iteration $(x,y)$
%can be computed in constant time.

The combination of the co-lexicographic order of the columns 
and the lexicographic order in which we traverse the grid has the following effect. 
If we add column $j'$ to column $j$, 
then $j' < j$ and thus $\gr(j') \leq \gr(j)$ in co-lexicographic order; 
furthermore, $j'$ must be in the image of the pivot map, 
and thus $j'$ has been visited and therefore 
$\gr(j') \leq \gr(j)$ in lexicographic order. 
It follows that $\gr(j') \leq \gr(j)$ in the usual partial order on $\mathbb{Z}^2$, 
which is crucial for the correctness of the algorithm.

Assuming $y \neq 1$, when we finish handling all columns at grade $(x,y-1)$, 
the sub-matrix of $A$ consisting of columns with grade less than or equal to $(x,y-1)$ 
(in the usual partial order on $\mathbb{Z}^2$) is reduced. 
Therefore, when we consider the grade $(x,y)$, 
in order to reduce the sub-matrix consisting of columns with grade less than or equal to $(x,y)$, 
we only need to reduce the columns with grade $(1,y), \dots, (x,y)$. 
See \cite{lesnick2019computing} for more details.

The sub-routine \texttt{Min\_gens} has no work to do on our running example, 
but \texttt{Ker\_basis} will have work to do, 
and the action of this sub-routine will illustrate the matrix reduction 
involved also in \texttt{Min\_gens}.

\paragraph{Details of \textup{\texttt{Ker\_basis}}}
This procedure is similar to the previous one, as it 
visits the columns in the same order, and reduces them when visiting.
There is one difference in the reduction procedure, however: every column
maintains an \define{auxiliary vector}.
Initially, the auxiliary vector of column $j$ is just the unit vector $e_j$
and whenever column $j'$ is added to column $j$, we also add the auxiliary vector of $j'$
to the auxiliary vector of $j$.
In linear algebra terms, the auxiliary vectors yield an auxiliary matrix $S$
(which is the identity matrix initially)
and letting $B'$
denote the matrix arising from $B$ at any point of the algorithm, 
we maintain the
invariant that $B'=BS$. 
In particular, if the $j$-th column of $B'$ is $0$, the $j$-th column
of $S$ encodes the linear combination of the columns of $B$
that represents a kernel element of the linear map $B$.

We describe the procedure \texttt{Ker\_basis} (Alg.~\ref{alg:ker_basis_lw}): 
Traverse the grid cells in lexicographic order.
When reaching grid cell $(x,y)$, iterate through all
matrix columns with grade $(1,y), (2,y),\ldots, (x,y)$ (in that order)
and reduce them as described above.
If any of these columns turn from non-zero to zero during the reduction,
append the auxiliary vector of the column to the output matrix
and set the grade of this column to $(x,y)$.
%(which is not necessarily the grade of the currently visited column in $B$).
The resulting matrix encodes the kernel basis of $B$. Its rows
correspond to the columns of $B$ and thus inherit their grades, yielding
a graded matrix as output.

In order to apply \texttt{Ker\_basis} to the matrix $[\partial^1]$ of our running example 
(displayed in \cref{sec:background}), 
we need to put the columns in co-lexicographic order. 
We fix the order $AB, AC, BD, DC, BC, CE, DE$. 
After the matrix reduction, the columns corresponding to 
$DC, BC$, and $DE$ are reduced to zero, and the output of \texttt{Ker\_basis} is:
\[
\scalebox{0.75}{
$\begin{array}{c|ccc}
& (0,0) & (1,1) & (2,2)\\
\hline
(0,0) & 1 & 1 & 1\\
(0,0) & 1 & 1 & 1\\
(0,0) & 1 & 0 & 1\\
(0,0) & 1 & 0 & 0\\
(1,1) & 0 & 1 & 0\\
(2,2) & 0 & 0 & 1\\
(2,2) & 0 & 0 & 1
\end{array}$
}
\]
This gives us the basis 
$\{AB + BD + DC + CA, 
AB + AC + BC, 
AB + AC + BD + CE + DE\}$
for the kernel of $[\partial^1]$.

\paragraph{Details of \textup{\texttt{Reparam}}}
Let $G$ denote the result of \texttt{Min\_gens} and $K$
the result of \texttt{Ker\_basis}. Note that $G$ and $K$ have
the same number of rows, with consistent grades. Form the matrix
$(K \vert G)$ and reduce each column of $G$, using auxiliary vectors. It is guaranteed
that this turns the matrix into $(K \vert 0)$, and the auxiliary vectors of the 
columns of $G$ yields a graded matrix $M'$ which is the output of the
procedure (Alg.~\ref{alg:reparam}).

Returning to our running example, the image of $[\partial^2]$ is generated by 
$G = \{ \partial^2(ABC), \partial^2(BDC), \partial^2(DCE) \}$. 
Re-ex\-press\-ing these with respect to our basis for the kernel of $[\partial^1]$, 
we obtain the following semi-minimal presentation:
\[
\scalebox{0.75}{
$M' = \begin{array}{c|ccc}
& (2,1) & (1,2) & (2,2)\\
\hline
(0,0) & 0 & 1 & 1\\
(1,1) & 1 & 1 & 0\\
(2,2) & 0 & 0 & 1
\end{array}$
}
\]

\paragraph{Details of \textup{\texttt{Minimize}}}
Let $n$ denote the number of columns of $M'$, the output of the previous step.
Traverse the columns of $M'$ from index $1$ to $n$.
If column $i$ is a local column (i.e., the grade of its pivot coincides
with the column grade), let $j$ denote its
pivot and iterate through the columns $i+1$ to $n$; if any column $k$ contains
row index $j$, add column $i$ to column $k$ (eliminating the row index at $j$).
At the end of this inner loop, no column except $i$ has a non-zero entry
at index $j$.
%As we explain in \ref{app:min_pres_algebra},
We can therefore remove column $i$ and row $j$ from the matrix,
without changing the persistence module that $M'$ presents.
So, remove column $i$ and row $j$ from the matrix.
After the outer loop has finished, re-index the remaining rows and columns, 
and return the resulting graded matrix $M$ as the minimal presentation (Alg.~\ref{alg:minimize_lw}). 
Note that the inner loop of \texttt{Minimize} can be parallelized.

Applying \texttt{Minimize} to the semi-minimal presentation $M'$ 
of our running example, 
we remove the single local column 
and obtain the following minimal presentation:
\[
\scalebox{0.75}{
$\begin{array}{c|cc}
& (2,1) & (1,2)\\
\hline
(0,0) & 0 & 1\\
(1,1) & 1 & 1
\end{array}$
}
\]

\subsection{Mpfree compression: Improvement via queues}\label{mpfree:Queues}

%We now describe methods to improve the performance of 
%\texttt{Min\_gens}, \texttt{Ker\_basis}, and \texttt{Minimize}. 
%In \cref{sec:experiments} we demonstrate the improvement in performance on 
%several classes of input data. 
We now describe methods to improve the performance of 
\texttt{Min\_gens} and \texttt{Ker\_basis}. 
We describe an improvement to \texttt{Minimize} in \cref{mpfree:Minimization}.  
In \cref{sec:experiments} we demonstrate the improvement in performance on 
several classes of input data.

%\paragraph{Queues}
%
In the worst case, the size of the grid used in \texttt{Min\_gens} and \texttt{Ker\_basis} 
is quadratic in the size of the input. 
It can happen, therefore, that simply storing and iterating over the grid can 
lead to quadratic behavior of the LW algorithm 
(see e.g. the ``Points on sphere'' data in \cref{sec:experiments}). 

One way around this is to control the size of the grid by coarsening: 
we choose a coarser grid in $\bbZ^2$ and snap each grade appearing in the matrices $A$ and $B$ 
to the closest grade in the coarser grid. 
Of course, this comes at the cost of computing an approximate solution to the original problem. 
See \cref{tbl:compare_with_rivet} in \cref{sec:experiments} 
for the effect of this coarsening on the LW algorithm. 
We now show that coarsening is not necessary, 
as \texttt{Min\_gens} and \texttt{Ker\_basis} can be adapted so that their performance is 
indifferent to the size of the grid.

First note that it is not sufficient to only consider the grades of the matrix columns 
in \texttt{Min\_gens} and \texttt{Ker\_basis}.
For instance, columns of $B$ on grades $(x',y)$ and $(x,y')$ with
$x'<x$ and $y'<y$ might combine into a kernel element at grade $(x,y)$,
so \texttt{Ker\_basis} has to perform work at grade $(x,y)$ even if
no column exists at this grade.

The main observation is that 
we can predict the grades on which the algorithm has
to (potentially) perform operations, 
effectively avoiding the iteration through all
grid cells. Surely, every grade that appears as a grade of matrix columns 
must be considered, to visit these columns for the first time. 
Moreover, consider the situation that the algorithm is at grade $(x,y)$ and reduces a column with index $i$.
Assume further that the pivot $j$ of $i$
appears already in $\rho$ for an index $k>i$. In that case, the LW-algorithm 
updates $\rho(j)\gets i$ and stops the reduction.
However, we know more: the next time that column $k$ is visited,
column $i$, or perhaps some other column with pivot $j$, will be added
to $k$. When is this next time? Since $i<k$ and the columns are in colex order,
we know that $y$, the $y$-grade of $i$, is smaller than or equal to $y'$, 
the $y$-grade of $k$. If $y'=y$, column $k$ will be handled in the same iteration, 
and nothing needs to be done. If $y'>y$, we know that the algorithm
needs to consider grade $(x,y')$.

Based on this idea, we set up a priority queue that stores the grid cells
that need to be visited, in lexicographic order. The queue is initialized
with the column grades of the matrix. 
Then, instead of iterating over all grid cells, the algorithms \texttt{Min\_gens}
and \texttt{Ker\_basis} keep
popping the smallest element from the queue until the queue is empty,
and proceed on each grade as described before.
We extend the reduction method of a column as follows (Alg.~\ref{alg:reduce_new}):
Whenever the algorithm encounters
a situation as above during a column reduction, it pushes $(x,y')$ 
to the queue.
Every element pushed to the queue is necessarily lexicographically
larger than the current element, so the algorithm terminates~-- in the worst
case after having handled every grid cell once, 
but skipping over many grid cells in practice.

\begin{proposition}
Using a priority queue for the grades as described, the algorithms 
\textup{\texttt{Min\_gens}} and 
\textup{\texttt{Ker\_basis}} produce the same output as in the LW algorithm.
\end{proposition}

\begin{proof}
The argument is mostly the same for \texttt{Min\_gens} and \texttt{Ker\_basis},
and we just talk about ``the algorithm'' for either of them. More precisely, we 
refer to the LW version of the algorithm and the optimized version of the algorithm
when talking about the variant without and with the priority queue, respectively.

We call a grade $(x,y)$ \emph{significant} if the algorithm, on grade $(x,y)$
performs a column operation on the matrix or appends a column to the output matrix.
We show that every significant grade in the LW algorithm is added to the priority
queue in the optimized version. That proves that the outcome is the same using an inductive argument.

Fix a grade $(x,y)$ and assume that the algorithm (in the LW version) performs
a column operation in this iteration. Let $i$ be the smallest index on which
such an operation is performed, let $j$ be its pivot, 
and let $(x',y)$ be its grade with $x'\leq x$.
If $x'=x$, then there is a matrix column with grade $(x,y)$ and the grade
is pushed into the grade priority queue initially. Otherwise, if $x'<x$,
column $i$ has been reduced previously in grade $(x-1,y)$, and the reduction
ended with $\rho(j)=i$ (otherwise, a further column addition would have been performed).
The fact that a column addition is needed at grade $(i,j)$ means that $\rho(j)$ must
have been re-set, to an index smaller than $i$. However, between $(x-1,y)$ and $(x,y)$,
the algorithms iterated over the grades $(x-1,y+1),(x-1,y+2),\ldots,(x-1,Y),(x,1),(x,2),\ldots (x,y-1)$
with $Y$ the maximal $y$-grade. Note that in the first part of the sequence, with $x$-grade $x-1$,
only columns with index $>i$ are updated because the matrix is stored in colex order.
Hence, none of these iterations can set $\rho(j)$ to a smaller index.
It follows that the update of $\rho(j)$ happens at a grade with $x$-grade $x$.
But then, the updated reduction algorithm ensures when $\rho(j)$ is updated,
the grade $(x,y)$ is added to the priority queue in this step.

Finally we consider the case that at grade $(x,y)$, the algorithm appends an output element. 
If the algorithm is \texttt{Min\_gens}, this only happens when the grade appears as matrix
column, and as argued earlier, these grades are added to the priority queue.
For \texttt{Ker\_basis}, an output column is added if a column is reduced from
non-zero to zero, and this implies that at least one column addition was performed, so
the grade is considered by the first part.
\end{proof}

A further improvement is based on a very similar idea:
note that when a grade $(x,y)$ is handled, 
both \texttt{Min\_gens} and \texttt{Ker\_basis}
still scan through all columns of grade $(x',y)$ with $x'\leq x$
and reduce all columns in this range.
Since only a few columns in this range typically need an update,
most of the time in the algorithm is wasted scanning through this range.

Necessary updates can be predicted during earlier steps in the algorithm:
as above, when, at grade $(x,y)$ 
column $i$ is reduced and its pivot is found in a column $k>i$,
we know that column $k$ needs an update. 
Let $y'\geq y$ be the $y$-grade of $k$.
We can just remember the index $k$ and handle it the next time
when $y$-grade $y'$ is visited, which will be at grade $(x,y')$. 

Technically, we realize this idea by storing one priority queue per $y$-grade.
In the extended column reduction, in a situation as above,
the index $k$ is pushed to the priority queue of its $y$-grade.
When handling a grade $(x,y)$, instead of 
scanning through the columns of grade $(1,y),\ldots,(x-1,y)$,
we keep popping the smallest index from
the priority queue of $y$ and reduce the column (this might introduce new
elements to the priority queue, if $y=y'$ with the notation from above, but new elements are of larger index, so the procedure
eventually empties the queue). After the queue is empty, the algorithm proceeds
with the columns on grade $(x,y)$ as in the LW-version. See Alg.~\ref{alg:min_gens_new}
and~\ref{alg:ker_basis_new} for pseudocode.

%It is not difficult to see that these variants perform exactly the same
%column operations as the original versions of \texttt{Min\_gens} and 
%\texttt{Ker\_basis}.
%
%For the second queuing optimization (to avoid the left-to-right scan at each grade),
%the proof strategy is very similar:
%
\begin{proposition}
Using one priority queue per $y$-grade as described, the algorithms 
\textup{\texttt{Min\_gens}}
and \textup{\texttt{Ker\_basis}} produce the same output as in the LW algorithm.
\end{proposition}

\begin{proof}
Again, we argue inductively that the same column operations are performed
in the variant without and with priority queues. Fix a column $i$
for which the LW variant (without priority queues) performs a column operation
for grade $(x,y)$. It suffices to show that $i$ is pushed to the priority queue
for $y$. If $i$ is at grade $(x,y)$, it is pushed by the algorithm as specified,
so we can assume that its grade is $(x',y)$ with $x'<x$.

Let $j$ be the pivot of column $i$. As in the previous proof,
the fact that a column addition is performed means that $\rho(j)$ has been updated
during the algorithm since $i$ was visited at grade $(x-1,y)$, and this can only
happen at a grade $(x,y')$ with $y'\leq y$. By the modified reduction method,
$i$ will be pushed into the priority queue of $y$ in this case.
\end{proof}

\subsection{Mpfree compression: Improvement to minimization}\label{mpfree:Minimization}
We now explain how the performance of \texttt{Minimize} can be improved. 
The modification is similar in spirit to the queue optimizations: 
we improve the runtime by avoiding unnecessary scanning for columns that need to be updated.

Recall that, whenever \texttt{Minimize} identifies
a local pair, the algorithm scans to the right to eliminate the local row index.
Simply performing this scan can lead to quadratic time complexity. 
Typically, the row index only appears in a few columns, and the scan
will query many columns that are not updated. Additionally, looking for a fixed
row index is a non-constant operation for most representations of column data.
For instance, if the column is realized as a dynamic array, it requires
a binary search per column.

We can avoid both scanning and binary search by not
eliminating local row indices immediately when they are identified as local. 
Instead, we simply apply \multichunk \,  to the short chain complex
$F^{1} \xrightarrow{M'} F^0 \to 0$, 
where $M'$ is the semi-minimal presentation produced by \texttt{Reparam}. 

\begin{proposition}
  If $P$ is a semi-minimal presentation of a persistence module~$W$,
  and $Q$ is the graded matrix obtained from $P$ by applying \textup{\multichunk},
  then $Q$ is a minimal presentation of $W$.
\end{proposition}

\begin{proof}
  Since the LW minimization algorithm produces a minimal presentation from $P$,
  it suffices to show that \multichunk \,  removes the same number of row-column pairs
  at each grade as the LW minimization algorithm.
  So, let $\textbf{z}$ be a grade of a column of $P$, and let $P_{\textbf{z}}$ be the sub-matrix of $P$
  consisting of the columns and rows of $P$ with grade $\textbf{z}$.
  To complete the proof, we show that 
  the number of row-column pairs with grade $\textbf{z}$ removed by the LW minimization algorithm
  and the number of row-column pairs with grade $\textbf{z}$ removed by \multichunk \,  
  are both equal to the rank of the matrix $P_{\textbf{z}}$.
  Say that a matrix $N$ has the ``distinct pivot'' property if any pair of non-zero columns of $N$ have
  distinct pivots. Note that if $N$ has the distinct pivot property, then the rank of $N$ is the number
  of non-zero columns of $N$.
  The column operations performed by the LW minimization algorithm 
  do not change the rank of $P_{\textbf{z}}$,
  and after performing all column operations, the sub-matrix $P_{\textbf{z}}$
  has the distinct pivot property, and the columns with grade $\textbf{z}$
  removed by the LW minimization algorithm are the non-zero columns of this matrix.
  Similarly, the column operations performed by 
  \multichunk \,  do not change the rank of $P_{\textbf{z}}$,
  and after performing all column operations, the sub-matrix $P_{\textbf{z}}$
  has the distinct pivot property, and the columns with grade $\textbf{z}$
  removed by \multichunk \,  are the non-zero columns of this matrix.
\end{proof}

\subsection{Mpfree compression: Complexity}\label{mpfree:Complexity}
It is easy to see that \mpfree \, has the same worst-case behavior
as the LW algorithm:
\begin{theorem}
Let $A$ be a $(b\times c)$-matrix and $B$ be an $(a\times b)$-matrix. 
Then, \textup{\mpfree} runs in
\[O((a+b+c)^3)\]
time and 
\[O(b(a+b+c))\]
space in the worst case.
\end{theorem}
\begin{proof}
The first three steps of \mpfree \, perform exactly the same matrix operations as the corresponding steps of the LW algorithm, and the complexity
bound (in terms of time and memory) is dominated by these matrices
and column operations on them. Hence, the same analysis as 
in~\cite[Sec 4.1]{lesnick2019computing} applies and yields the desired bound 
for the first three steps.
The additional costs of maintaining priority queues is negligible, because
the maximal number of operations on these queues is $O((a+b+c)^2)$, and
each operation only needs logarithmic time. Moreover, the total number of elements in all queues is $O(bc)$ for \texttt{Min\_gens} and $O(ab)$ for \texttt{Ker\_basis}, so also the memory bound is satisfied.

For the minimization procedure, note that the semi-minimal presentation
is a $(b\times c)$-matrix, and we apply \multichunk \, on it. 
In the notation of the complexity analysis of Section~\ref{sec:multi-chunk},
$n\gets b+c$, and $g$ and $\ell$ are both bounded by $n$. 
By Theorem~\ref{thm:multi-chunk-complexity} and the subsequent remark,
we get a worst-case complexity of $O(n^3)$ and space complexity of $O(n^2)$
which proves the theorem.
\end{proof}

To make our improvements ``visible'' in a complexity analysis, we have
to express the complexity of the algorithm using different parameters. 
We focus on the time complexity.
Let $M'$ denote the total number of column additions during the algorithm,
and $M$ denote the total cost to perform them. 
Note that $M'$ is quadratic in the input size
in the worst case, and $M=O(nM')$ which has led to our cubic 
bound in the previous theorem.
However, in practice, we can expect that $M$ and $M'$ are much smaller 
than what the worst case predicts, similar to the one-parameter case.
Revisiting the LW algorithm, we can express the complexity 
of \texttt{Min\_gens} and \texttt{Ker\_basis}
as $O(M+XY)$,
where $X$ and $Y$ are the grid dimensions of the instance. 
This is because the subroutines iterate over all cells of the grid. In particular, the $O(XY)$ factor appears in practice in every instance,
which is visible for input with a large grid size.

In contrast, \mpfree \, rather has a complexity of $O(M+M'\log n)$
because in the worst case, there is one priority queue operation triggered
by every column operation. Note that if $M$ and $M'$ are linear in $n$
(which is to be expected for realistic inputs), the algorithm shows a complexity
of $O(n\log n)$, independently of the grid size.

\medskip

For the minimization procedure, we still denote by $M$ 
the total cost of all column additions. Then, the LW algorithm
runs its minimization in $O(M+n^2 \log n)$ time, where the second
terms is due to the scanning to the right, looking for a row index to be cleared.
Note again that this scan is performed every time a row index is removed
so we have to expect a quadratic time behavior on every instance.
In contrast, \mpfree \, uses \multichunk \, to avoid the
scan. We can bound its complexity by $O(M+\Delta)$, where $\Delta$ is the cost
of the compression step apart from the column additions performed.\footnote{We remark that the $M$ in this step
is not exactly equal to the $M$ in the LW algorithm because the columns additions
are not identical. Experimentally, we observed that the two costs
are of similar size.} 
$\Delta$ is proportional to the number of non-zero entries encountered
in the presentation matrix, which can be quadratic in the worst case if
the matrix becomes dense. However, this is not to be expected in most cases,
and assuming that the matrix remains sparse and $M$ is linear as before,
the minimization step shows linear behavior.

\section{Experimental results}
\label{sec:experiments}
We implemented both described algorithms in C++ and provide them as software
repositories~\cite{multi-chunk-repo, mpfree-repo} 
%\multichunk\footnote{\url{https://bitbucket.org/mkerber/multi_chunk/src/master/}} and 
%\mpfree\footnote{\url{https://bitbucket.org/mkerber/mpfree/src/master/}} 
under the LGPL license.

\paragraph{Software design}
A few notable aspects of our code:

Both algorithms accept inputs in the \texttt{scc2020} format~\cite{scc2020} and write
their result in the same format. \multichunk \, reads a chain complex
of arbitrary length $k$ and outputs a chain complex of same length,
potentially with fewer generators. \mpfree \, reads a chain complex
of length $2$ and outputs a chain complex of length $1$, encoding
the minimal presentation.

Both algorithms come as a sequential version and a shared-memory parallelized
version using the \texttt{OpenMP} library. This is achieved just by adding the
line
\begin{verbatim}
#pragma omp parallel for schedule(guided,1)
\end{verbatim}
in front of the loops to be run in parallel (see the documentation of \texttt{OpenMP} for the details of the command). We have parallelized the compression
procedure of \multichunk, the reparameterization and the minimization
procedure of \mpfree, and some parts of the input processing.
Unless stated otherwise, we will use the parallelized version throughout
our experiments.

Our algorithms are \emph{generic} in the sense that they operate on
any graded matrix with a well-defined interface using the template mechanism
of C++. This allows our code, in particular, to run on every matrix representation that is provided by the \texttt{PHAT} library. In~\cite{kerber2021fast},
we compared all choices and found out that the \texttt{vector\_vector}
representation was generally a good choice, so we use that representation
throughout our experiments.

\paragraph{Experimental setup}
Our code was compiled with gcc-7.5.0,
and ran on a workstation with an Intel(R) Xeon(R) CPU E5-1650 v3 CPU (6 cores, 12 threads, 3.5GHz) and 64 GB RAM, running GNU/Linux (Ubuntu 16.04.5).
We measured the overall time and memory consumption using the linux
command \texttt{/usr/bin/time}, and the \texttt{timer} library of the
\textsc{Boost} library to measure the time for subroutines.

We used the versions of \multichunk, \mpfree, and RIVET 
available in their bitbucket/github repositories on 6 July 2022.

%Code that was used to generate benchmark files is 
%available\footnote{\url{https://bitbucket.org/mkerber/cgta_paper_2021/}} 
%though we omit the actual input files themselves because of their size. 
The scripts for generating the input data and running the experiments, 
as well as the input files, are available 
\cite{tugraz-data-repo}. 
This repository also contains the output files of the experiments, 
which contain the data on time and memory consumption that is summarized 
in the tables in this section.

To test the performance, we consider six different classes of input data
of different sizes, resulting in a total of around 50 GB of data (uncompressed). 
In each case the number of parameters $d$ is $2$. 
With one exception, these data sets are artificial; our goal was to model
input types that appear in practice and to still have some control over the
input sizes. For each combination of class and parameters, we generated
$5$ instances and display the average in the tables. 
Our classes are in detail:

\begin{description}
\item[Full function Rips] 
We sample $n$ points from a noisy circle and consider the simplicial
complex consisting of all $n$ points, $e=\binom{n}{2}$ edges
and $f=\binom{n}{3}$ triangles. 
The first coordinate of the bi-grade
is $0$ for vertices, the distance between endpoints for edges, 
and the length of the longest edge for triangles.
This is the \emph{Vietoris-Rips filtration}, one of the most prominent filtrations
in topological data analysis. For the second grade coordinate,
we used (the negative of) a kernel density estimate with Gaussian kernel with fixed bandwidth
for the vertices, and this is extended to edges and triangles
by assigning the maximal grade among the boundary vertices.
We take the negative of the density estimate so that points in high-density regions 
appear at smaller grades than points in low-density regions. 
The construction ensures that the subcomplex at each grade
is a \emph{flag complex}, that is, a complex that is completely
determined by its vertices and edges.
This yields a chain complex of length $2$ with generators $f,e,n$.
Note that the number of generators is very imbalanced in this example,
and the large number of triangles restricts the number of vertices
(our largest instance has 289 points).

\item[Function Rips with threshold]
We sample $n$ points on a sphere in $\bbR^3$ (with non-uniform density)
and fix a target number $N$ of simplices. We then build the Vietoris-Rips
filtration on these vertices up to $3$-simplices, until it consists of $N$
simplices. The second coordinate of the bi-grade is determined as in
the previous case, except that a linear kernel was used for efficiency.
We choose $N$ as large as 16 million, and always choose $n$ to be $0.05\cdot N$, so that $5\%$ of the simplices are vertices. This leads to chain complexes
of length $3$ with many more generators at level $0$ (up to 800,000) and
a more balanced distribution of generators across the levels.

\item[Points on sphere]
We generate the convex hull of $n$ points 
on a sphere in $\bbR^3$ sampled uniformly at random.
By Euler's formula, the convex hull consists of exactly
$e=3n-6$ edges and $f=2n-4$ triangles.
The grades of a vertex are simply its $x$- and $y$-coordinates
(ignoring the $z$-coordinate), and the $x$-/$y$-grades
of an edge and a triangle is just the maximal $x$-/$y$-coordinate
among its boundary vertices
%(the fact that the grades are in $\bbR^2$ rather than $\Z^2$
%does not make a difference because there are only finitely many of them and so
%they can be mapped bijectively to $\Z^2$ preserving orders).
This is also known as the \emph{lower star (bi-)filtration}
which generalizes a commonly used filtration type
for a single parameter
and has been used as a benchmark example in previous multi-parameter work
\cite{Scaramuccia2018}.
Our largest instances consist of $1.6$ million vertices.

\item[Random Delaunay triangulations]
As a variant of the previous example, 
we sample $n$ points inside the unit sphere,
and we compute the Delaunay triangulation of these $n$ points (using \textsc{CGAL}~\cite{cgal:delaunay}).
As before, we assign to vertices their $x$- and $y$-coordinates as grades,
and to edges, triangles, and tetrahedra using the maximal grade of the boundary vertices.
This results in a chain complex of length $3$ with rather balanced numbers
of generators. Our largest instances consists of $640,000$ vertices.

\item[Multi-cover bi-filtration]
We generated bi-filtrations coming from multi-covers of point clouds: 
the simplicial complex at grade $(r,k)$ represents the region of the plane
covered by at least $k$ balls of radius $r$ around the input points.
A recent result~\cite{cklo-computing} shows how to compute this structure
efficiently based on rhomboid tiling~\cite{eo-multi}.
In our examples, the maximal $k$ is set to $10$. 
Running the LW algorithm, this leads 
to a grid with only $10$ different grades in the $y$-direction.
Also, when picking points in $\bbR^3$, the rhomboid tiling is $4$-dimensional,
resulting in a chain complex of length $4$. Our largest instances pick $390$
points in $\bbR^3$, resulting in chain complexes with around 16 million
generators.

\item[Off-datasets]
We also tested with triangular mesh data,
publicly available at the AIM@SHAPE repository\footnote{available at \url{http://visionair.ge.imati.cnr.it/}}. 
The graded input matrices $(A,B)$ were generated in the same way as for the previously described sphere and random Delaunay meshes.
In this experiments, we restrict to the four datasets ``hand'', ``eros'', ``dragon'' and ``raptor'', each yielding a chain complex of length $2$
with $0.16$, $2.4$, $3.2$ and $5.0$ million simplices, respectively. 
\end{description}

\paragraph{Comparison with RIVET}
We demonstrate that our improvements over the original algorithm
by Lesnick and Wright~\cite{lesnick2019computing} are highly effective.
For that, we present a comparison of \mpfree \, with the minimal
presentation algorithm of RIVET. By comparing the numbers for \mpfree \, and RIVET in Table~\ref{tbl:compare_with_rivet},
we see a clear improvement in time and memory.

We remark that both the usage of priority queues and of using \multichunk \, for the minimization procedure lead to substantial speed-ups.
This was demonstrated carefully in the conference version~\cite{kerber2021fast}. We cannot re-do these experiments because \mpfree \, has no option
to switch off these optimizations.

The main purpose of RIVET is the visualization of a 2-parameter module;
for that reason, the input grades are usually snapped to an $x\times y$-grid
before further processing. While the implementation of the LW-algorithm
in RIVET does not require snapping, it has not been optimized
for large grids.
Therefore, we ran the two algorithms on the same instances
after snapping all bi-grades to a (uniformly chosen) $50\times 50$ grid.
As expected, RIVET becomes significantly faster and more memory efficient.
However, \mpfree \, is still faster even without snapping the input.
In fact, \mpfree's performance after snapping improves by a factor of less than $2$ 
in terms of time and memory for all inputs, so we do not include it
in Table~\ref{tbl:compare_with_rivet}. 

\begin{table*}[h!]\centering
\begin{tabular}{c|l||cc|cc||cc}
& & \multicolumn{4}{c||}{Original} & \multicolumn{2}{c}{Snapped}\\
& & \multicolumn{2}{c|}{RIVET} & \multicolumn{2}{c||}{mpfree} & \multicolumn{2}{c}{RIVET}\\
Class & N & Time & Memory & Time & Memory & Time & Memory \\
\hline
\multirow{3}{*}{FFR} & 253K& 13.8& 98MB& 0.58& 74MB& 11.5& 97MB\\
& 508K& 30.4& 205MB& 1.29& 148MB& 25.9& 223MB\\
& 1.02M& 71.5& 455MB& 3.14& 302MB& 58.9& 503MB\\
\hline
\multirow{3}{*}{PoS} & 37K& 16.9& 3.11GB& 0.12& 31MB& 1.92& 24MB\\
& 75K& 68.0& 12.4GB& 0.25& 64MB& 5.32& 48MB\\
& 150K& 381& 49.3GB& 0.57& 138MB& 16.9& 103MB\\
\hline
\multirow{3}{*}{MCF} & 40K& 1.81& 25MB& 0.11& 25MB& 1.75& 24MB\\
& 100K& 9.40& 79MB& 0.29& 69MB& 8.18& 81MB\\
& 199K& 34.7& 205MB& 0.61& 156MB& 28.6& 212MB\\
\hline
OFF & 172K & $>$3600 & 35.3GB & 0.77 & 145MB & 8.59 & 193MB \\
\end{tabular}
\caption{Time and memory comparison with the RIVET library.
From left to right: class of input, number of generators, running time (in seconds) and memory consumption (three times).}
\label{tbl:compare_with_rivet}
\end{table*}

\paragraph{Performance of \textup{\multichunk}}
See Table~\ref{tbl:bench_multi_chunk} for the performance of \multichunk \, (parallel version). Note that the instances
are much larger than for the RIVET comparison from Table~\ref{tbl:compare_with_rivet}. 
Generally, we can say
that the running time of \multichunk \, is dominated 
by the time for reading the input and converting it into a sequence of boundary matrices
(displayed in the column ``IO''). The compression rate shows the relative
file size of the output file in terms of the file size of the input 
(e.g., 5\% means that the input file is $20$ times bigger than the output).
We observe that the rate differs
over instance classes, but can reduce the chain complex by a factor of more 
than $10$ for many instances. We remark that the output complex might well
be bigger even if the total number of generators does not increase because
of potential fill-up of the boundary matrices. While this happens
for the full function Rips datasets, we point out that passing to this
(slightly) larger chain complex is still beneficial for computing
minimal presentations, as we demonstrate later in the experiments.
Finally, we see a nearly linear behavior in time and memory for all instance classes. 

\begin{table*}[!h]\centering
\begin{tabular}{c|c|ccc|c}
Class & N & Time & IO & Memory & Compression rate \\
\hline
\multirow{4}{*}{FFR}
 & 0.51M & 2.39 & 92\% & 292MB & 132\% \\
 & 1.00M & 5.07 & 91\% & 628MB & 158\% \\
 & 2.00M & 10.2 & 91\% & 1.25GB & 160\% \\
 & 4.02M & 22.3 & 89\% & 2.72GB & 188\% \\
\hline
\multirow{4}{*}{FRT}
 & 2M & 6.5 & 96\% & 1.05GB & 23\% \\
 & 4M & 13.7 & 95\% & 2.08GB & 23\% \\
 & 8M & 28.1 & 96\% & 4.17GB & 23\% \\
 & 16M & 57.5 & 96\% & 8.43GB & 23\% \\
\hline
\multirow{4}{*}{PoS}
 & 2M & 6.18 & 97\% & 0.90GB & 26\% \\
 & 4M & 12.7 & 97\% & 1.74GB & 25\% \\
 & 8M & 25.9 & 97\% & 3.45GB & 26\% \\
 & 16M & 51.5 & 97\% & 6.80GB & 26\% \\
\hline
\multirow{4}{*}{RDT}
 & 2.29M & 5.19 & 95\% & 1.01GB & 1.1\% \\
 & 4.59M & 10.9 & 93\% & 1.99GB & 1.3\% \\
 & 9.22M & 23.4 & 91\% & 3.95GB & 1.7\% \\
 & 18.5M & 51.4 & 86\% & 7.95GB & 2.4\% \\
\hline
\multirow{4}{*}{MCF}
 & 2.02M & 4.36 & 95\% & 0.97GB & 6.7\% \\
 & 4.01M & 9.21 & 95\% & 1.90GB & 7.0\% \\
 & 8.10M & 19.6 & 95\% & 3.83GB & 7.8\% \\
 & 16.1M & 40.9 & 95\% & 7.49GB & 8.8\% \\
\hline
\multirow{4}{*}{OFF}
 & 0.17M & 0.43 & 96\% & 80.3MB & 23\% \\
 & 2.38M & 5.5 & 97\% & 1.00GB & 26\% \\
 & 3.28M & 8.39 & 97\% & 1.33GB & 26\% \\
 & 5.00M & 13.8 & 97\% & 2.11GB & 24\% \\
\end{tabular}
\caption{Results for \multichunk \, on large datasets. 
From left to right: class of input, number of generators, running time (in seconds), 
percentage of running time used for IO operations, memory consumption, and compression rate.}
\label{tbl:bench_multi_chunk}
\end{table*}

\paragraph{Combining \textup{\multichunk} and \textup{\mpfree}}
On the same instances as in the previous paragraph, we also ran \mpfree. More precisely we computed all minimal presentations induced by the chain complex, that is, for any pair of consecutive maps,
we applied \mpfree \, independently. The output is then a sequence of $d-1$ minimal presentations for a chain complex of length $d$. 
We measured both the running time of the sequential and the parallel
version of \mpfree. 
Table~\ref{tbl:bench_mpfree_with_multi_chunk} shows the timings
on the left (``mpfree only''). We observe a super-linear behaviour in running time and a close-to-linear behavior in memory for most examples. Also, we can see that parallelization only gives marginal speed-ups, if at all.

Futhermore, we combined both algorithms by first applying
\multichunk \, on the input chain complex and subsequently applying 
\mpfree \, on the compressed chain complex, with the same workflow
as above. As we observe in Table~\ref{tbl:bench_mpfree_with_multi_chunk},
the success of preprocessing with \multichunk \, varies from class to class,
but is sometimes significant, in particular for the multi-cover
and the random Delaunay data. One aspect is the length of the chain complex:
the \multichunk \, preprocessing can make use of the clearing optimization
to avoid many column operations that would otherwise be performed
in the subroutines \texttt{Min\_gens} and \texttt{Ker\_basis}.
But also for short chain complexes where clearing does not apply,
\multichunk \, reduces the time and memory consumption. In particular,
the performance of the combined algorithm is coming much closer to a linear
behaviour. 

The last column of Table~\ref{tbl:bench_mpfree_with_multi_chunk} shows the compression rate 
as the ratio of the file sizes of all minimal presentations combined over the size of the input 
(as before, 5\% means that the input file is $20$ times bigger than the output). 
We see that indeed, the minimal presentations
are significantly smaller than the input chain complex.

While in the sequential version,
\multichunk \, does not help for full function Rips filtrations (FFR), it still
leads to a time improvement with parallelization (at the cost
of using more memory). The reason is that most time in these instances
is spent on column reductions in the dominant highest dimension and the
\multichunk \, allows for a parallel processing of these steps.
It should be remarked, however, that the execution is done with $12$ threads,
and the speed-up is far away from the ideal factor of $12$.

\begin{table*}[!h]\centering
{\footnotesize
\begin{tabular}{c|c||cc|cc||cc|cc||c}
& & \multicolumn{4}{c||}{mpfree only} & \multicolumn{4}{c||}{mpfree and multi-chunk} \\
& & \multicolumn{2}{c|}{sequential} & \multicolumn{2}{c||}{parallel} & \multicolumn{2}{c|}{sequential} & \multicolumn{2}{c||}{parallel} \\
Class & N & Time & Memory & Time & Memory & Time & Memory & Time & Memory & Size\\
\hline
\multirow{4}{*}{FFR}
 & 0.51 & 2.01 & 0.32 & 2.07 & 0.32 & 2.1 & 0.41 & 1.60 & 0.42 & 0.02\% \\
 & 1.00 & 4.64 & 0.63 & 4.82 & 0.63 & 4.89 & 0.89 & 3.42 & 0.92 & 0.01\% \\
 & 2.00 & 9.84 & 1.24 & 10.0 & 1.24 & 10.1 & 1.79 & 7.11 & 1.83 & 0.01\% \\
 & 4.02 & 24.3 & 2.49 & 24.6 & 2.49 & 25.2 & 3.91 & 16.0 & 4.02 & 0.01\% \\
\hline
\multirow{4}{*}{FRT}
 & 2 & 13.0 & 1.52 & 12.7 & 1.62 & 8.05 & 1.13 & 7.51 & 1.16 & 9.4\% \\
 & 4 & 29.4 & 2.92 & 28.7 & 3.23 & 17.3 & 2.21 & 16.1 & 2.32 & 9.4\% \\
 & 8 & 65.8 & 5.98 & 63.8 & 6.46 & 37.6 & 4.42 & 35.1 & 4.61 & 9.3\% \\
 & 16 & 145 & 12.4 & 140 & 13.3 & 80.7 & 8.80 & 74.8 & 9.23 & 9.3\% \\
\hline
\multirow{4}{*}{PoS}
 & 2 & 12.6 & 2.91 & 11.9 & 3.55 & 6.77 & 1.10 & 6.48 & 1.17 & 0.04\% \\
 & 4 & 27.5 & 6.34 & 25.8 & 7.76 & 14.4 & 2.19 & 13.7 & 2.35 & 0.03\% \\
 & 8 & 60.7 & 13.9 & 55.7 & 17.46 & 30.4 & 4.41 & 28.8 & 4.76 & 0.02\% \\
 & 16 & 134 & 30.8 & 122 & 39.13 & 63.7 & 8.83 & 60.0 & 9.76 & 0.02\% \\
\hline
\multirow{4}{*}{RDT}
 & 2.29 & 25.5 & 2.35 & 24.1 & 2.69 & 5.92 & 0.99 & 5.22 & 1.01 & 0.7\% \\
 & 4.59 & 69.5 & 5.11 & 65.0 & 5.79 & 13.9 & 1.92 & 11.1 & 1.99 & 0.9\% \\
 & 9.22 & 191 & 11.14 & 176 & 12.84 & 35.2 & 3.81 & 23.9 & 3.95 & 1.2\% \\
 & 18.5 & 555 & 24.63 & 500 & 28.93 & 98.4 & 7.59 & 52.6 & 7.96 & 1.8\% \\
\hline
\multirow{4}{*}{MCF}
 & 2.02 & 37.8 & 3.23 & 36.6 & 3.68 & 5.41 & 0.95 & 5.1 & 0.97 & 3.4\% \\
 & 4.01 & 112 & 7.50 & 108 & 8.75 & 11.7 & 1.85 & 10.8 & 1.91 & 3.7\% \\
 & 8.10 & 330 & 17.4 & 317 & 20.7 & 26.7 & 3.77 & 23.7 & 3.92 & 4.3\% \\
 & 16.1 & 927 & 38.0 & 889 & 46.2 & 61.5 & 7.67 & 52.6 & 7.90 & 5.0\% \\
\hline
\multirow{4}{*}{OFF}
 & 0.17 & 0.87 & 0.18 & 0.82 & 0.20 & 0.43 & 0.09 & 0.41 & 0.09 & 0.5\% \\
 & 2.38 & 19.4 & 3.20 & 18.6 & 3.73 & 6.53 & 1.16 & 6.13 & 1.23 & 0.1\% \\
 & 3.28 & 26.0 & 4.01 & 24.6 & 4.63 & 10.3 & 1.63 & 9.69 & 1.74 & 0.3\% \\
 & 5.00 & 39.0 & 5.90 & 37.3 & 6.81 & 15.2 & 2.35 & 14.3 & 2.48 & 0.2\% \\
\end{tabular}
} %small
\caption{Results for running \mpfree \, and combining \mpfree \, with \multichunk, sequential and parallel version. From left to right: class of input, number of generators (in millions), running time (in seconds) and memory consumption (in GB) (four times). The last column denotes the total size of all minimal presentation files (in bytes) divided by the size of the input size (also in bytes).}
\label{tbl:bench_mpfree_with_multi_chunk}
\end{table*}

\smallskip

We also report on a perhaps surprising observation. For that, we truncated
chain complexes $C_n\to \ldots \to C_1\to C_0\to 0$ of length $>2$ (classes FRT, RDT, and MCF) to a short
chain complex $C_2\to C_1\to C_0$.
The truncated chain complex is smaller in size, saving time on
IO operations, and \mpfree \, needs to be called only once.
Nevertheless, as we see in Table~\ref{tbl:truncated}, the total running time
is still sometimes larger for the truncated chain complex, compared to computing all minimal presentations using \multichunk \, (compare the RDT instances in Tables~\ref{tbl:bench_mpfree_with_multi_chunk} and~\ref{tbl:truncated}). 
The reason for this
slow-down is that for the full chain complex, \multichunk \, can use
the clearing optimization to avoid the reduction of the majority of columns
in the first map of the chain complex, whereas this information is lost in
the truncated complex, and the algorithm has to fully reduce all columns to zero.

\begin{table*}[!h]\centering
{\footnotesize
\begin{tabular}{c|c||cc|cc||cc|cc}
& & \multicolumn{4}{c||}{mpfree only} & \multicolumn{4}{c}{mpfree and multi-chunk}\\
& & \multicolumn{2}{c|}{sequential} & \multicolumn{2}{c||}{parallel} & \multicolumn{2}{c|}{sequential} & \multicolumn{2}{c}{parallel} \\
Class & N & Time & Memory & Time & Memory & Time & Memory & Time & Memory\\
\hline
\multirow{4}{*}{FRT}
 & 1.12 & 6.52 & 1.01 & 6.33 & 1.12 & 4.50 & 0.71 & 4.37 & 0.74 \\
 & 2.24 & 14.9 & 2.05 & 14.1 & 2.30 & 10.1 & 1.40 & 9.62 & 1.47 \\
 & 4.48 & 33.3 & 4.24 & 31.5 & 4.63 & 22.0 & 2.80 & 20.9 & 2.97 \\
 & 8.97 & 72.9 & 8.92 & 68.4 & 9.76 & 46.9 & 5.68 & 43.8 & 6.05 \\
\hline
\multirow{4}{*}{RDT}
 & 1.68 & 10.8 & 1.63 & 9.88 & 1.82 & 13.9 & 0.97 & 5.68 & 1.08 \\
 & 3.37 & 27.3 & 3.49 & 24.0 & 3.90 & 45.6 & 2.19 & 14.3 & 2.51 \\
 & 6.75 & 70.1 & 7.46 & 58.0 & 8.45 & 161 & 5.19 & 40.7 & 6.16 \\
 & 13.5 & 190 & 16.3 & 141 & 18.7 & 604 & 13.2 & 131 & 16.2 \\
\hline
\multirow{4}{*}{MCF}
 & 1.23 & 9.99 & 1.20 & 9.76 & 1.33 & 5.61 & 0.63 & 4.52 & 0.67 \\
 & 2.44 & 26.4 & 2.50 & 25.9 & 2.78 & 13.3 & 1.26 & 10.6 & 1.33 \\
 & 4.93 & 70.4 & 5.91 & 69.1 & 6.62 & 32.5 & 2.57 & 25.8 & 2.77 \\
 & 9.76 & 176 & 13.1 & 172 & 14.9 & 77.3 & 5.27 & 60.6 & 5.74 \\
\end{tabular}
} %small
\caption{Results for running \mpfree \, and combining \mpfree \, with \multichunk, sequential and parallel version. Note that the instances here are subsets of the corresponding full chain complexes displayed
in Table~\ref{tbl:bench_mpfree_with_multi_chunk}. From left to right: class of input, number of generators (in millions), running time (in seconds) and memory consumption (in GB) (four times).}
\label{tbl:truncated}
\end{table*}

\paragraph{Detailed runtime analysis}
We list the running time of the major substeps of the algorithm in Table~\ref{tbl:detailed_analysis}. We restrict to the generally best configuration, namely
running \multichunk \, before \mpfree, with parallelization.

We observe that, similar to running~\multichunk \, alone, a substantial
part of the running time goes into reading the input file and building up
the matrices, speaking in favor of the efficiency of our algorithm.
We observe that depending on the instance, the more expensive steps
of the method are the compression subroutine of \multichunk,
and the subroutines \texttt{Min\_Gens} and \texttt{Ker\_basis}
of \mpfree. Without parallelization, the increase
in running time (visible in Table~\ref{tbl:bench_mpfree_with_multi_chunk})
is mostly due to the increased running time of the compression subroutine.
When \multichunk \, is not used, \texttt{Min\_Gens} and \texttt{Ker\_basis} get more expensive, resulting in a further increase in runtime.
The conference version~\cite{kerber2021fast} contains some precise numbers
which we omit for brevity.

\begin{table*}[!h]\centering
\begin{tabular}{c|c|c||cccccccc}
Class & N & Time & IO & LR & CP & MG & KB & RP & Min\\
\hline
\multirow{4}{*}{FFR}
 & 0.51 & 1.60 & 0.88 & 0.04 & 0.14 & 0.30 & 0.00 & 0.00 & 0.00 \\
 & 1.00 & 3.42 & 1.74 & 0.08 & 0.37 & 0.75 & 0.00 & 0.00 & 0.00 \\
 & 2.00 & 7.11 & 3.44 & 0.15 & 0.76 & 1.79 & 0.00 & 0.00 & 0.00 \\
 & 4.02 & 16.0 & 7.00 & 0.30 & 1.97 & 4.70 & 0.00 & 0.00 & 0.00 \\
\hline
\multirow{4}{*}{FRT}
 & 2 & 7.51 & 5.10 & 0.12 & 0.12 & 0.45 & 0.56 & 0.12 & 0.24 \\
 & 4 & 16.1 & 10.4 & 0.27 & 0.24 & 1.12 & 1.44 & 0.26 & 0.56 \\
 & 8 & 35.1 & 21.8 & 0.51 & 0.45 & 2.78 & 3.69 & 0.55 & 1.36 \\
 & 16 & 74.8 & 44.7 & 1.03 & 0.91 & 6.87 & 8.97 & 1.13 & 3.26 \\
\hline
\multirow{4}{*}{PoS}
 & 2 & 6.48 & 4.62 & 0.05 & 0.12 & 0.17 & 0.98 & 0.04 & 0.02 \\
 & 4 & 13.7 & 9.56 & 0.10 & 0.23 & 0.46 & 2.26 & 0.10 & 0.04 \\
 & 8 & 28.8 & 19.5 & 0.20 & 0.43 & 1.28 & 5.16 & 0.20 & 0.10 \\
 & 16 & 60.0 & 39.2 & 0.39 & 0.85 & 4.02 & 11.6 & 0.42 & 0.23 \\
\hline
\multirow{4}{*}{RDT}
 & 2.29 & 5.22 & 4.64 & 0.09 & 0.14 & 0.03 & 0.03 & 0.02 & 0.02 \\
 & 4.59 & 11.1 & 9.67 & 0.19 & 0.46 & 0.06 & 0.07 & 0.03 & 0.04 \\
 & 9.22 & 23.9 & 20.3 & 0.38 & 1.60 & 0.16 & 0.18 & 0.09 & 0.10 \\
 & 18.5 & 52.6 & 42.2 & 0.77 & 5.97 & 0.39 & 0.43 & 0.23 & 0.28 \\
\hline
\multirow{4}{*}{MCF}
 & 2.02 & 5.1 & 3.98 & 0.10 & 0.10 & 0.19 & 0.29 & 0.07 & 0.07 \\
 & 4.01 & 10.8 & 8.23 & 0.19 & 0.22 & 0.47 & 0.75 & 0.16 & 0.18 \\
 & 8.10 & 23.7 & 17.4 & 0.36 & 0.49 & 1.27 & 2.11 & 0.36 & 0.47 \\
 & 16.1 & 52.6 & 36.5 & 0.74 & 1.14 & 3.46 & 6.11 & 0.86 & 1.14 \\
\hline
\multirow{4}{*}{OFF}
 & 0.17 & 0.41 & 0.32 & 0.00 & 0.01 & 0.00 & 0.04 & 0.00 & 0.00 \\
 & 2.38 & 6.13 & 4.08 & 0.05 & 0.10 & 0.15 & 1.30 & 0.05 & 0.02 \\
 & 3.28 & 9.69 & 6.35 & 0.08 & 0.16 & 0.29 & 2.09 & 0.08 & 0.04 \\
 & 5.00 & 14.3 & 10.4 & 0.12 & 0.27 & 0.27 & 2.15 & 0.05 & 0.04 \\
\end{tabular}
\caption{Results for running \multichunk \, and \mpfree, parallel version. From left to right: class of input, number of generators (in millions), running time (in seconds), time for IO operations (IO), time for local reductions in \multichunk \, (LR), time for compression in \multichunk \, (CP), time for \texttt{Min\_gens} in \mpfree \, (MG), time for \texttt{ker\_basis} (KB), time for reparameterization (RP), time for minimization (Min)}
\label{tbl:detailed_analysis}
\end{table*}

\section{Conclusions}
\label{sec:conclusions}
We presented the two compression techniques \multichunk \, and \mpfree \, which preserve homological properties
of the underlying data sets with significant compression rates
and fast performance.

There is no doubt that compression is an inevitable step in making
multi-parameter persistent homology a practical tool for data analysis.
We hope that upcoming algorithmic work on $2$-parameter filtrations
will report in more detail on the performance gains of compression
techniques.

We expect our software to be immediately useful for practioners trying
to analyze $2$-parameter filtrations. Nevertheless, there are many 
improvements possible: one idea we have not carried over from the $1$-parameter
setup is the \emph{dualization} or \emph{cohomology} approach, 
which suggests inverting the chain complex and proceeding with increasing
indices. This has proved to be a successful strategy in particular
for Vietoris-Rips complexes~\cite{Bauer2017, ripser} but is harder to apply
for $2$ or more parameters. Recently, a result in this direction
has been announced~\cite{lbl-efficient} and we hope our speed-ups will be
complementary to their approach.
Better compression rates and/or performance gains could be achieved
for special types of filtrations; 
for example, a result for multi-parameter flag filtrations has been announced 
recently~\cite{akp-flag}. 
Finally, further gains could be achieved by combination approximation techniques,
yielding compressed structures that are not equivalent, but provably close
to the input.

\bibliographystyle{ieeetr}
\bibliography{ref.bib}

\appendix
\section{Detailed proofs}
\label{sec:app_details_1}

We begin by completing the proof of \cref{local-pair-homotopy-equivalence}.

\begin{proposition}
 The maps $r^\bullet$ and $s^\bullet$ are homotopy-inverse one with respect to the other.
\end{proposition}

\begin{proof}
Let us define a collection of maps $\phi^l: F^l \rightarrow F^{l+1}$, for any $l\in \bbN$.
Given $c \in B^l$, the map $\phi^l$ is defined as:
\[\phi^l(c)=\begin{cases}
 \langle c, c_1 \rangle c_2 & \text{ if } l=n,\\
0 & \text{ otherwise.}
\end{cases}\]
The result follows from the following facts, which are straightforward to check. 
For any $l$:
  \begin{itemize}
    \item[(1)] $s^l r^l$ is chain-homotopic to $\mathrm{id}_{F^l}$ via $\phi^l$,
    \item[(2)] $r^l s^l = \mathrm{id}_{\bar{F}^l}$.
  \end{itemize}
\end{proof}

We now complete the proofs of the results of \cref{chunk:Optimal}. 
Let $A^\bullet, B^\bullet$ be two chain complexes such that $A^\bullet$ is a subcomplex of $B^\bullet$ and let $\eta^n: H^n(A^\bullet)\rightarrow H^n(B^\bullet)$ be the map in homology induced by this inclusion.
Let us consider the boundary matrix of $B^\bullet$ in which the generators of $A^\bullet$ are placed on the left with respect to the generators of $B^\bullet$ not belonging to $A^\bullet$. By applying the standard reduction algorithm for computing persistent homology~\cite{Edelsbrunner2010}, we obtain a reduced matrix $R$.
To avoid a clash in terminology with previous parts of the paper,
we refer to the entry of lowest position in the matrix among the non-zero ones in column $c'$ as the \emph{lowest index} of $c'$.
Let us denote the columns of $R$ corresponding to the generators of $B^\bullet$ not belonging to $A^\bullet$ as the columns of $B^\bullet\setminus A^\bullet$ and the others as the columns of $A^\bullet$.

\begin{lemma}\label{lemma:matrixReduction}
  We have that:
  \begin{itemize}
    \item the number of the $n$-columns of $B^\bullet\setminus A^\bullet$ having as lowest index a column of $A^\bullet$ is equal to $\dim\,\Ker \, \eta^{n-1}$,
    \item the number of the $n$-columns of $B^\bullet\setminus A^\bullet$ with null boundary for which there is no column in $R$ having them as lowest index is equal to $\dim\,\Coker \, \eta^n$.
  \end{itemize}
\end{lemma}

\begin{proof}
We can partition the $n$-columns of $B^\bullet\setminus A^\bullet$ in three distinct classes.
The first class includes the $n$-columns of $B^\bullet\setminus A^\bullet$ having as lowest index a column of $A^\bullet$. According to the standard reduction algorithm, these columns are in one-to-one correspondence with the non-zero homology classes of $H^{n-1}(A^\bullet)$ that become trivial in $H^{n-1}(B^\bullet)$. So, their number coincides with the dimension of $\Ker \, \eta^{n-1}$.
The second class includes the $n$-columns of $B^\bullet\setminus A^\bullet$ with null boundary for which there is no column in $R$ having them as lowest index.
According to the standard reduction algorithm, these columns are in one-to-one correspondence with the homology classes that are born in $H^{n}(B^\bullet)$. So, their number coincides with the dimension of $\Coker \, \eta^{n}$.
The third class of $n$-columns of $B^\bullet\setminus A^\bullet$ are the ones with lowest index in $B^\bullet\setminus A^\bullet$ or the columns with null boundary which are the lowest index of a column of $B^\bullet\setminus A^\bullet$.
\end{proof}

\begin{theorem}\label{thm:optimality2}
Let $G^{\bullet}$ be the output of \textup{\multichunk} applied to the input $F^{\bullet}$. Then $\delta_{\bfp}^n(F^{\bullet})=\gamma_{\bfp}^n(G^{\bullet})$.
\end{theorem}

\begin{proof}
Let us apply Lemma~\ref{lemma:matrixReduction} to the case $A^\bullet=F_{<\bfp}^\bullet$ and $B^\bullet=F_{\bfp}^\bullet$. Note that the local reduction phase
at grade $\bfp$ in \multichunk \, corresponds to a partial reduction
of the matrix of $F_{\bfp}^\bullet$, because all column additions are left-to-right,
and it is known that the lowest indices of the reduced matrix are invariant
under left-to-right column additions. At the end of the local reduction phase,
each column is labeled as local (positive or negative), or global.
A local column is paired with another local column of grade $\bfp$,
and this pairing will not change when completing the partial reduction
to a completely reduced matrix for $F_{\bfp}^\bullet$.
The reason is simply that all columns
in $F_{<\bfp}^\bullet$ have a lowest index with smaller index
than any column of grade $\bfp$.
By the above reasoning, the local columns will neither contribute to
kernel nor cokernel of $\eta_{\bfp}^{n-1}$.

On the other hand, consider a global column with grade $\bfp$. Its lowest index
corresponds to a column in $F_{<\bfp}^\bullet$ after the local reduction phase.
Since the lowest index can only decrease when further reducing the matrix,
it follows that in the reduced matrix, the global column has
still a lowest index in $F_{<\bfp}^\bullet$, in which case it contributes to $\Ker\,\eta_{\bfp}^{n-1}$, or its boundary becomes null. In the latter case, it cannot be paired with
another column at grade $\bfp$ since then, this pairing would have been
identified already at the local reduction phase. Hence, the global column
contributes to $\Coker \, \eta_{\bfp}^{n}$.
After all, each global column contributes to the kernel or the cokernel,
and therefore, the number of global columns is equal to $\delta_{\bfp}^n(F^{\bullet})$.
Moreover, the number of global columns at grade $\bfp$ is exactly the number
of generators added at grade $\bfp$ to the output complex $G^{\bullet}$.
\end{proof}

%Given an arbitrary bi-filtered chain complex $\bar{F}$ and $p=(p_x, p_y)\in \bbZ^2$, let us consider the chain complex $\bar{F}_{<p}^\bullet$ defined as $\sum_{q<p} \bar{F}_q^\bullet$.
%Since we are considering just finitely generated chain complexes, there exists a sufficiently small $\epsilon$ for which, by defining $p_1=(p_x-\epsilon, p_y)$ and $p_2=(p_x, p_y-\epsilon)$, we have that, for any $p_1\leq q \leq p$, $\bar{F}_q^\bullet=\bar{F}_{p_1}^\bullet$, and, for any $p_2\leq q \leq p$, $\bar{F}_q^\bullet=\bar{F}_{p_2}^\bullet$. Then, $\bar{F}_{<p}^\bullet$ can be expressed as the finite sum $\bar{F}_{p_1}^\bullet + \bar{F}_{p_2}^\bullet$.
%Moreover, defining $p_0$ as the value $(p_x-\epsilon, p_y-\epsilon)$, the one-criticality of the bi-filtration $\bar{F}$ ensures that $\bar{F}_{p_0}^\bullet=\bar{F}_{p_1}^\bullet \cap \bar{F}_{p_2}^\bullet$.
%For any arbitrary bi-filtered chain complex $\bar{F}$, let $\eta_p^n$ denote the homology map in dimension $n$ induced by the inclusion of $\bar{F}_{<p}^\bullet$ into $\bar{F}_p^\bullet$.
Now, let $\bar{F}^{\bullet}$ be any chain complex of finitely generated free persistence modules. 
Since each $n$-generator newly introduced in $\bar{F}_{\bfp}^{\bullet}$ can act on the homology by just increasing the dimension of $\Ker \, \eta_{\bfp}^{n-1}$, or alternatively of $\Coker \, \eta_{\bfp}^n$, at most by 1, we have that
\begin{equation}\label{eq:ker-coker}
\delta_{\bfp}^n(\bar{F}^{\bullet})=\dim\,\Ker \, \eta_{\bfp}^{n-1} + \dim\,\Coker \, \eta_{\bfp}^n \leq \dim\,\bar{F}_{\bfp}^n - \dim\,\bar{F}_{<\bfp}^n=\gamma_{\bfp}^n(\bar{F}^{\bullet}).
\end{equation}

\begin{theorem}\label{lemma:leq2}
Any chain complex of free persistence modules $\bar{F}^{\bullet}$ quasi-isomorphic to $F^{\bullet}$ has to add at least $\delta_{\bfp}^n(F^{\bullet})$ $n$-generators at grade $\bfp$. I.e., $\delta_{\bfp}^n(F^{\bullet})\leq \gamma_{\bfp}^n(\bar{F}^{\bullet})$.
\end{theorem}

\begin{proof}
% Let $\bar{F}$ be a bi-filtered chain complex satisfying the hypothesis.
Let $f^\bullet: F^{\bullet} \rightarrow \bar{F}^{\bullet}$ be the quasi-isomorphism between $F^{\bullet}$ and $\bar{F}^{\bullet}$.
Thanks to their commutativity with the inclusion maps, we have an induced chain map $f_{<\bfp}^\bullet: F_{<\bfp}^{\bullet} \rightarrow \bar{F}_{<\bfp}^{\bullet}$.
This ensures for any $\bfp \in \bbZ^2$, the construction of the following commutative diagram
\[
\begin{tikzcd}
&0  \arrow[r] &F_{\bfp_0}^n \arrow[r] \arrow[d, "f_{\bfp_0}^n"] &F_{\bfp_1}^n \oplus F_{\bfp_2}^n \arrow[r] \arrow[d, "{(f_{\bfp_1}^n, f_{\bfp_2}^n)}"] &F_{<\bfp}^n \arrow[r]\arrow[d, "f_{<\bfp}^n"] &0 \\
&0  \arrow[r] &\bar{F}_{\bfp_0}^n \arrow[r] &\bar{F}_{\bfp_1}^n \oplus \bar{F}_{\bfp_2}^n \arrow[r] &\bar{F}_{<\bfp}^n \arrow[r] &0 \\
\end{tikzcd}
\]
in which each row is exact.
Thanks to the quasi-isomorphism established by $f$ and to the Mayer-Vietoris sequence, we can derive from the previous one the following commutative diagram

{\tiny
% {\scriptsize
\[
\begin{tikzcd}
&H^n(F_{p_0}^\bullet)  \arrow[r] \arrow[d, "\psi_{p_0}^n"] &H^n(F_{p_1}^\bullet) \oplus H^n(F_{p_2}^\bullet)  \arrow[r] \arrow[d, "{(\psi_{p_1}^n, \psi_{p_2}^n)}"] &H^{n}(F_{<p}^\bullet)  \arrow[r] \arrow[d, "\psi_{<p}^n"] &H^{n-1}(F_{p_0}^\bullet)  \arrow[r] \arrow[d, "\psi_{p_0}^{n-1}"] &H^{n-1}(F_{p_1}^\bullet) \oplus H^{n-1}(F_{p_2}^\bullet) \arrow[d, "{(\psi_{p_1}^{n-1}, \psi_{p_2}^{n-1})}"] \\
&H^n(\bar{F}_{p_0}^\bullet)  \arrow[r] &H^n(\bar{F}_{p_1}^\bullet) \oplus H^n(\bar{F}_{p_2}^\bullet)  \arrow[r] &H^{n}(\bar{F}_{<p}^\bullet)  \arrow[r] &H^{n-1}(\bar{F}_{p_0}^\bullet)  \arrow[r] &H^{n-1}(\bar{F}_{p_1}^\bullet) \oplus H^{n-1}(\bar{F}_{p_2}^\bullet) \\
\end{tikzcd}
\]
}
in which each row is exact and the horizontal and the vertical maps are induced by the corresponding ones between chain complexes.
Since all the other four vertical maps are isomorphisms, the $5$-lemma ensures that map $\psi_{<\bfp}^n$ is also an isomorphism. So, we have the following commutative diagram

\[
\begin{tikzcd}
&H^n(F_{<\bfp}^\bullet) \arrow[r, "\iota_{\bfp}^n"] \arrow[d, "\psi_{<\bfp}^n"'] 
&H^n(F_{\bfp}^\bullet) \arrow[d, "\psi_{\bfp}^n"] \\
&H^n(\bar{F}_{<\bfp}^\bullet) \arrow[r, "\eta_{\bfp}^n"'] 
&H^n(\bar{F}_{\bfp}^\bullet) \\
\end{tikzcd}
\]

in which vertical maps are isomorphisms.
Then, for any $n$, $\dim\,\Ker \, \eta_{\bfp}^{n}=\dim\,\Ker \, \iota_{\bfp}^{n} $ and $\dim\,\Coker \, \eta_{\bfp}^n=\dim\,\Coker \, \iota_{\bfp}^n$.
So, by applying Equation (\ref{eq:ker-coker}), we have that
\begin{align*}
\delta_{\bfp}^n(F^{\bullet}) 
= \dim\,\Ker \, \iota_{\bfp}^{n-1} + \dim\,\Coker \, \iota_{\bfp}^n 
&= \dim\,\Ker \, \eta_{\bfp}^{n-1} + \dim\,\Coker \, \eta_{\bfp}^n \\
&\leq \dim\,\bar{F}_{\bfp}^n - \dim\,\bar{F}_{<\bfp}^n 
= \gamma_{\bfp}^n(\bar{F}^{\bullet}).
\end{align*}
\end{proof}

\section{Pseudocode}
\label{app:pseudocode}

In this section, we give pseudocode for some of the algorithms discussed in the paper.  
Algorithm~\ref{alg:locality} describes the locality test, while Algorithms~\ref{alg:phaseI}, \ref{alg:phaseII}, and \ref{alg:phaseIII} describe the three phases (local reduction, compression, and removal of local columns) of \multichunk.
Algorithms~\ref{alg:reduce_lw}--\ref{alg:minimize_lw}
describe the subroutines of the LW algorithm, 
and Algorithms~\ref{alg:reduce_new}--\ref{alg:ker_basis_new} 
describe our improved versions. 
We mark the for loops in the pseudocode that can be run in parallel.

\begin{algorithm*}
\caption{Locality test}
\label{alg:locality}
\begin{algorithmic}
\Function{is\_local}{$A,i$}
    \If {column $i$ in $A$ is $0$}
      \State{\Return {$true$}}
    \EndIf
    \State {$j\gets$ pivot of column $i$ in $A$}
    \State {Return true iff the column grade of $i$ in $A$ equals the row grade of $j$ in $A$}
\EndFunction
\end{algorithmic}
\end{algorithm*}

\begin{algorithm*}
\caption{Local reduction}
\label{alg:phaseI}
\begin{algorithmic}
\Function{Local\_Reduction}{$A$, $LocPos$, $LocNeg$, $Global$}
    \LeftComment{$A$, a graded boundary matrix $[\partial^{n}]^{B^n, B^{n-1}}$ of a chain complex $F^\bullet$}
    \LeftComment{$LocPos$, $LocNeg$, $Global$, subsets of the generators of $F^\bullet$}
    \For{$c$ column of $A$ and not in $LocPos \cup LocNeg \cup Global$}
    \LeftComment{Columns are traversed in increasing order w.r.t. index $i$}
    \While{\Call{is\_local}{A,c} and $\exists c'$ column of $A$ with $i(c')<i(c)$}
    \State{$c\gets c+ c'$}
    % with $\lambda$ such that the local pivot of $c$ disappears}
    \EndWhile
    \If{\Call{is\_local}{A,c}}
    \State{Add $c$ to $LocNeg$}
    \State{Add the local pivot $c'$ of $c$ to $LocPos$}
    \Else
    \State{Add $c$ to $Global$}
    \EndIf
    \EndFor
    \State{\Return{$LocPos$, $LocNeg$, $Global$}}
\EndFunction
\end{algorithmic}
\end{algorithm*}

\begin{algorithm*}
\caption{Compression}
\label{alg:phaseII}
\begin{algorithmic}
\Function{Compression}{$A$, $LocPos$, $LocNeg$, $Global$}
    \LeftComment{$A$, a graded boundary matrix $[\partial^{n}]^{B^n, B^{n-1}}$ of a chain complex $F^\bullet$}
    \LeftComment{$LocPos$, $LocNeg$, $Global$, subsets of the generators of $F^\bullet$}
    \For{$c$ column of $A$ and in $Global$}
    \While{boundary of $c$ contains an elements of $LocNeg\cup LocPos$}
    \State{Pick local generator $c'$ in the boundary of $c$ with maximal index}
    \If{$c'\in LocNeg$}
    \State{Remove $c'$ from the boundary of $c$}
    \Else
    \Comment{$c'\in LocPos$}
    \State{Denote as $c''$ the $n$-column in $LocNeg$ with $c'$ as local pivot}
    \State{$c\gets c+ c''$}
    % with $\lambda$ such that $c'$ disappears in the boundary of $c$}
    \EndIf
    \EndWhile
    \EndFor
    \State{\Return{$A$}}
\EndFunction
\end{algorithmic}
\end{algorithm*}

\begin{algorithm*}
\caption{Removal of local pairs}
\label{alg:phaseIII}
\begin{algorithmic}
\Function{Local\_Pairs\_Removal}{$A$, $LocPos$, $LocNeg$, $Global$}
    \LeftComment{$A$, a graded boundary matrix $[\partial^{n}]^{B^n, B^{n-1}}$ of a chain complex $F^\bullet$}
    \LeftComment{$LocPos$, $LocNeg$, $Global$, subsets of the generators of $F^\bullet$}
    \For{$c$ column of $A$}
    \If{$c \in LocNeg\cup LocPos$}
    \State{Remove $c$ from $A$}
    \EndIf
    \EndFor
    \For{$r$ row of $A$}
    \If{$r \in LocNeg\cup LocPos$}
    \State{Remove $r$ from $A$}
    \EndIf
    \EndFor
    \State{\Return{$A$}}
\EndFunction
\end{algorithmic}
\end{algorithm*}

\begin{algorithm*}
\caption{Matrix reduction, LW-version}
\label{alg:reduce_lw}
\begin{algorithmic}
\Function{reduce\_LW}{$A,j,use\_auxiliary=false$}
    \LeftComment{$A$ maintains a pivot vector $piv$}
    \LeftComment{$piv[i]=k$ means that column $k$ in $A$ has pivot $i$.}
    \LeftComment{$piv[i]=-1$ means that no (visited) column in $A$ has pivot $i$.}
    \LeftComment{If \textit{use\_auxiliary} is set, $A$ also maintains an auxiliary matrix.}
    \While{column $j$ in $A$ is not empty}
    \State {$i\gets$ pivot of column $j$}
    %\If {\texttt{local} and the row-grade of $j$ differs from the column-grade of $i$ in $A$}
    %\State {\textbf{break}}
    %\EndIf
    \State {$k\gets piv[i]$}
    \If {$k=-1$ or $k>j$}
    \State $piv[i]\gets j$
    \State {\textbf{break}}
    \Else
    \State {add column $k$ to column $j$}
    \If {\textit{use\_auxiliary}}
    \State {add auxiliary-column $k$ to auxiliary-column $j$}
    \EndIf
    \EndIf
    \EndWhile
\EndFunction
\end{algorithmic}
\end{algorithm*}

\begin{algorithm*}
\caption{Min\_gens, LW-version}
\label{alg:min_gens_lw}
\begin{algorithmic}
\Function{Min\_gens\_LW}{$A$}
    \State {$(X,Y) \gets grid\_dim\_of(A)$}\Comment{grades of $A$ are on $X\times Y$ grid}
    \State {$Out\gets\emptyset$}
    \For {$x=1,\ldots,X$}
    \For {$y=1,\ldots,Y$}
    \State {$L\gets$column indices of $A$ with grades $(0,y),\ldots(x-1,y)$ in order}
    \For {$i\in L$}
    \State {\Call{Reduce\_LW}{A,i}}
    \EndFor
    \State {$I\gets$ column indices of $A$ with grade $(x,y)$}
    \For {$i\in I$}
    \State {\Call{Reduce\_LW}{A,i}}
    \If {column $i$ is not zero}
    \State {Append column $i$ to $Out$ with grade $(x,y)$}
    \EndIf
    \EndFor
    \EndFor
    \EndFor
    \State{\Return{Out}}
\EndFunction
\end{algorithmic}
\end{algorithm*}

\begin{algorithm*}
\caption{Ker\_basis, LW-version}
\label{alg:ker_basis_lw}
\begin{algorithmic}
\Function{Ker\_basis\_LW}{$B$}
    \State {$(X,Y) \gets grid\_dim\_of(B)$}\Comment{grades of $B$ are on $X\times Y$ grid}
    \State {$Out\gets\emptyset$}
    \For {$x=1,\ldots,X$}
    \For {$y=1,\ldots,Y$}
    \State {$L\gets$column indices of $B$ with grades $(0,y),\ldots(x,y)$ in order}
    \For {$i\in L$}
    \State {\Call{Reduce\_LW}{B,i,use\_auxiliary=true}}
    \If {column $i$ has turned from non-zero to zero}
    \State {Append auxiliary-column of $i$ to $Out$ with grade $(x,y)$}
    \EndIf
    \EndFor
    \EndFor
    \EndFor
    \State{\Return {$Out$}}
\EndFunction
\end{algorithmic}
\end{algorithm*}

\begin{algorithm*}
\caption{Reparameterize}
\label{alg:reparam}
\begin{algorithmic}
\Function{Reparam}{$G,K$}
    \LeftComment{$G$ is the output of $\Call{Min\_gens}{A}$}
    \LeftComment{$K$ is the output of $\Call{Ker\_basis}{B}$}
    \State {Form matrix $(K|G)$ by concatenation.}
    \State {$Out\gets\emptyset$}
    \For {$i$ in index range of $G$-column in $(K|G)$} \Comment{Parallelizable}
    \State {\Call{Reduce\_LW}{$(K|G),i,use\_auxiliary=true$}}\Comment{Column $i$ is zero afterwards}
    \State {Append the auxiliary column $i$ to $Out$}
    \EndFor
    \State{\Return {$Out$}}
\EndFunction
\end{algorithmic}
\end{algorithm*}

\begin{algorithm*}
\caption{Minimization, LW-version}
\label{alg:minimize_lw}
\begin{algorithmic}
\Function{Minimize\_LW}{$M'$}
\State {$n\gets\text{\# columns in $M'$}$} 
\For {$i=1,\ldots,n$} 
\If{\Call{is\_local}{$M',i$}}
\State {$j\gets$ pivot of column $i$ in $M'$}
\For{$k=i+1,\ldots,n$}
\If{column $k$ in $M'$ contains $j$ as row index}
\State{Add column $i$ to column $k$}
\EndIf
\EndFor
\State {Mark column $i$ and row $j$ in $M'$}
\EndIf
\EndFor
\State {$M\gets$ submatrix of $M'$ consisting of unmarked rows and columns}
\State {Re-index the columns of $M$ and \Return{$M$}}
\EndFunction
\end{algorithmic}
\end{algorithm*}

%%%% Our version

\begin{algorithm*}
\caption{Matrix reduction, new version}
\label{alg:reduce_new}
\begin{algorithmic}
\Function{Reduce\_New}{$A,i,grade,use\_auxiliary=false,local\_check=false$}
    \LeftComment{$grade$ is the grade that the algorithm is currently considering}
    \LeftComment{On top of pivot vector and auxiliary vector, $A$ maintains priority queues:}
    \LeftComment{$cols\_at\_y\_grade[y_0]$ is a priority queue for column indices at $y$-grade $y_0$}
    \LeftComment{$grade\_queue$ is a priority queue of grades that the algorithm needs to look at}
    \While{column $i$ in $A$ is not empty}
    \State {$j\gets$ pivot of column $i$}
    \State {$k\gets piv[j]$}
    \If {$k=-1$}
    \State $piv[j]\gets i$
    \State {\textbf{break}}
    \EndIf
    \If {$k>j$}
    \State {$y\gets$ $y$-grade of column $k$ in $A$}
    \State {Push $k$ to $cols\_at\_y\_grade[y]$}
    \State {$x\gets$ $x$-grade of $grade$}
    \State {Push $(x,y)$ to $grade\_queue$}
    \State $piv[j]\gets i$
    \State {\textbf{break}}
    \EndIf
    \If {\texttt{local\_check} and not \Call{is\_local}{A,i}}
    \State {\textbf{break}}
    \EndIf
    \State {add column $k$ to column $i$}
    \If {\textit{use\_auxiliary}}
    \State {add auxiliary-column $k$ to auxiliary-column $i$}
    \EndIf
    \EndWhile
\EndFunction
\end{algorithmic}
\end{algorithm*}

\begin{algorithm*}
\caption{Min\_gens, new version}
\label{alg:min_gens_new}
\begin{algorithmic}
\Function{Min\_gens\_New}{$A$}
    \State {Initialize $grade\_queue$ as empty priority queue}
    \For {each column $c$ of $A$}
    \State{Push the grade of $c$ into $grade\_queue$}
    \EndFor
    \For {each $y$-grade $y_0$ that appears in a column of $A$}
    \State {Initialize $cols\_at\_y\_grade[y_0]$ as empty priority queue}
    \EndFor
    \State {$Out\gets\emptyset$}
    \While{$grade\_queue$ is not empty}
    \State{Pop the grade $(x,y)$ from $grade\_queue$ minimal in lex order}
    \State{Push all column indices at grade $(x,y)$ into $cols\_at\_y\_grade[y]$}
    \While{$cols\_at\_y\_grade[y]$ is not empty}
    \State{Pop column index $i$ from $cols\_at\_y\_grade[y]$ minimal in lex order}
    \State{\Call{Reduce\_new}{A,i,(x,y)}}\Comment{Might change the priority queues as side effect}
    \If{column $i$ is not $0$ and its grade is $(x,y)$}
    \State {Append column $i$ to $Out$ with grade $(x,y)$}
    \EndIf
    \EndWhile
    \EndWhile
    \State{\Return{Out}}
\EndFunction
\end{algorithmic}
\end{algorithm*}

\begin{algorithm*}
\caption{Ker\_basis, new version}
\label{alg:ker_basis_new}
\begin{algorithmic}
\Function{Ker\_basis\_new}{$B$}
    \State {Initialize $grade\_queue$ as empty priority queue}
    \For {each column $c$ of $B$}
    \State{Push the grade of $c$ into $grade\_queue$}
    \EndFor
    \For {each $y$-grade $y_0$ that appears in a column of $B$}
    \State {Initialize $cols\_at\_y\_grade[y_0]$ as empty priority queue}
    \EndFor
    \State {$Out\gets\emptyset$}
    \While{$grade\_queue$ is not empty}
    \State{Pop the grade $(x,y)$ from $grade\_queue$ minimal in lex order}
    \State{Push all column indices at grade $(x,y)$ into $cols\_at\_y\_grade[y]$}
    \While{$cols\_at\_y\_grade[y]$ is not empty}
    \State{Pop column index $i$ from $cols\_at\_y\_grade[y]$ minimal in lex order}
    \State{\Call{Reduce\_new}{B,i,(x,y),use\_auxiliary=true}}\Comment{Might change the priority queues as side effect}
    \If{column $i$ has turned from non-zero to zero}
    \State {Append auxiliary-column $i$ to $Out$ with grade $(x,y)$}
    \EndIf
    \EndWhile
    \EndWhile
    \State{\Return{Out}}
\EndFunction
\end{algorithmic}
\end{algorithm*}

\end{document}